\documentclass[12pt]{amsart}

\usepackage{
  amsmath,
  amsfonts,
  amssymb,
  amsthm,
  eufrak,
  amscd,
  comment,
  paralist,
  etoolbox,
  eucal,
  mathtools,
  enumitem,
  mathdots,
  verbatim,
  extarrows,    
}
\usepackage[usenames,dvipsnames]{xcolor}
\usepackage[all]{xy}
\usepackage[makeroom]{cancel}

\usepackage{bbm}                     
\usepackage[colorlinks=true, linkcolor=blue, citecolor=blue, urlcolor=blue, breaklinks=true]{hyperref}

\usepackage[capitalize]{cleveref}   

\crefname{defin}{Definition}{Definitions}
\crefname{eg}{Example}{Examples}
\crefname{lem}{Lemma}{Lemmas}
\crefname{theo}{Theorem}{Theorems}
\crefname{equation}{}{}
\crefname{enumi}{}{}
\newcommand{\rom}[1]{\uppercase\expandafter{\romannumeral #1\relax}}

\newcommand\N{\mathbb{N}}

\newcommand\kk{\Bbbk}
\newcommand\one{\mathbbm{1}}

\newcommand\cC{\mathcal{C}}
\newcommand\cS{\mathcal{S}}
\newcommand\cR{\mathcal{R}}
\newcommand\cB{\mathcal{B}}
\newcommand\cD{\mathcal{D}}

\newcommand\PR{\mathcal{PR}}
\newcommand\RB{\mathcal{RB}}
\newcommand\TL{\mathcal{TL}}

\newcommand\cM{\mathcal{M}}
\newcommand\op{\text{op}}
\newcommand\rev{\text{rev}}
\newcommand\Par{\mathcal{P}\mathit{ar}}

\DeclareMathOperator{\End}{End}
\DeclareMathOperator{\Hom}{Hom}

\DeclareMathOperator{\id}{id}
\DeclareMathOperator{\Ob}{Ob}

\usepackage{tikz}
\usetikzlibrary{arrows,patterns}
\usepackage{tikz-cd}

\newcommand{\pd}[1]{\filldraw[black] (#1) circle (1.5pt)} 

\newcommand{\braidto}{to[out=up,in=down]}

\tikzset{anchorbase/.style={>=To,baseline={([yshift=-0.5ex]current bounding box.center)}}}

\newtheorem{theo}{Theorem}[section]

\newtheorem{prop}[theo]{Proposition}
\newtheorem{lem}[theo]{Lemma}
\newtheorem{cor}[theo]{Corollary}

\theoremstyle{definition}
\newtheorem{defin}[theo]{Definition}
\newtheorem{rem}[theo]{Remark}

\numberwithin{equation}{section}
\allowdisplaybreaks

\renewcommand{\theenumi}{\alph{enumi}}
\setenumerate[1]{label=(\alph*)}          

\newtoggle{comments}
\newtoggle{details}
\newtoggle{detailsnote}

\iftoggle{comments}{%
  \newcommand{\acomments}[1]{
    \ \\
    {\color{red}
      \textbf{AS:} #1
    }
    \ \\
    }
  \newcommand{\scomments}[1]{
    \ \\
    {\color{red}
      \textbf{SNL:} #1
    }
    \ \\
    }
}{%
  \newcommand{\acomments}[1]{}
  \newcommand{\scomments}[1]{}
}

\iftoggle{details}{%
  \newcommand{\details}[1]{
      \ \\
      {\color{OliveGreen}
        \textbf{Details:} #1
      }
      \\
  }
}{%
  \newcommand{\details}[1]{}
}

\begin{document}
%

\title{Presentations of diagram categories}

\author{Mengwei Hu}
\address{
  Department of Mathematics and Statistics \\
  University of Ottawa \\
  Ottawa, ON K1N 6N5, Canada 
}
\address{
  College of Mathematics \\
  Sichuan University \\
  Chengdu, Sichuan, 610065, China
}
\email{mengwei.hu0616@gmail.com}

\begin{abstract}
  We describe the planar rook category, the rook category, the rook-Brauer category, and the Motzkin category in terms of generators and relations. We show that the morphism spaces of these categories have linear bases given by planar rook diagrams, rook diagrams, rook-Brauer diagrams, and Motzkin diagrams.
\end{abstract}



\ifboolexpr{togl{comments} or togl{details}}{%
  {\color{magenta}DETAILS OR COMMENTS ON}
}{%
}

\maketitle
\thispagestyle{empty}

\tableofcontents
\section{Introduction}\label{introduction}

Let $t$ be an element in an arbitrary commutative ring $\kk$. Let $k,l$ be two nonnegative integers. The partition algebra $P_k(t)$ arose in the work of P. Martin \cite{Mar94} and later, independently, in the work of V. Jones \cite{Jon94}. The partition algebra $P_k(t)$ has a linear basis given by all partition diagrams of type $\binom{k}{k}$. A \emph{partition} of type $\binom{l}{k}$ is a set partition of $\{1,\cdots,k,1',\cdots,l'\}$, and the elements of the partition are called \emph{blocks}. For convenience, we represent a partition of type $\binom{l}{k}$ by a diagram called a \emph{partition diagram} with $k$ vertices in the bottom row, labeled by $1,\cdots,k$ from left to right, and $l$ vertices in the top row, labeled by $1',\cdots,l'$ from left to right. We draw edges connecting vertices so that the connected components of the diagram are the blocks of the partition. For example,
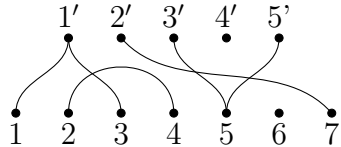
\[
  \begin{tikzpicture}[anchorbase]
    \pd{0.7,1} node[anchor=south] {$1'$};
    \pd{1.4,1} node[anchor=south] {$2'$};
    \pd{2.1,1} node[anchor=south] {$3'$};
    \pd{2.8,1} node[anchor=south] {$4'$};
    \pd{3.5,1} node[anchor=south] {$5$'};  
    \pd{0,0} node[anchor=north] {$1$};
    \pd{0.7,0} node[anchor=north] {$2$};
    \pd{1.4,0} node[anchor=north] {$3$};
    \pd{2.1,0} node[anchor=north] {$4$};
    \pd{2.8,0} node[anchor=north] {$5$};
    \pd{3.5,0} node[anchor=north] {$6$};
    \pd{4.2,0} node[anchor=north] {$7$};
    \draw (0,0) \braidto (0.7,1);
    \draw (1.4,0) \braidto (0.7,1);
    \draw (0.7,0) to[out=up,in=up,looseness=1.5] (2.1,0);
    \draw (2.8,0) \braidto (2.1,1);
    \draw (2.8,0) \braidto (3.5,1);
    \draw (4.2,0) ..controls (3.5,0.8) and (2.1,0.2).. (1.4,1);
  \end{tikzpicture}
\]
is a partition diagram of type $\binom{5}{7}$. Composition of partition diagrams is given by vertical stacking; see \cref{partitioncategorysubsection} for more details. A partition is \emph{planar} if it can be represented as a diagram without edge crossings inside of the rectangle formed by its vertices. There are many interesting subalgebras of $P_k(t)$.
\begin{align*}
  PP_k(t) &= \{D\in P_k(t)\ |\ D\ \text{is planar}\},\\
  RB_k(t) &= \{D\in P_k(t)\ |\ \text{all blocks of }D\text{ have size at most 2}\},\\
  M_k(t) &= RB_k(t)\cap PP_k(t),\\
  B_k(t) &= \{D\in P_k(t)\ |\ \text{all blocks of}\ D\ \text{have size 2}\},\\
  TL_k(t) &= B_k(t)\cap PP_k(t),\\
  R_k(t) &= RB_k(t)\cap\{D\in P_k(t)\ |\ \text{vertices in the same block of}\ D\ \text{lie in distinct rows}\},\\
  PR_k(t) &=R_k(t)\cap PP_k(t).
\end{align*}
The subalgebras $RB_k(t)$, $M_k(t)$, $B_k(t)$, $TL_k(t)$, $R_k(t)$, $PR_k(t)$ are called the \emph{rook-Brauer algebra}, the \emph{Motzkin algebra}, the \emph{Brauer algebra}, the \emph{Temperley-Lieb algebra}, the \emph{rook algebra}, and the \emph{planar rook algebra} respectively. Similarly, we can give the definitions of \emph{rook-Brauer diagrams}, \emph{Motzkin diagrams}, \emph{Brauer diagrams}, \emph{Temperley-Lieb diagrams}, \emph{rook diagrams}, and \emph{planar rook diagrams}, which do not require the bottom rows and the top rows have equal numbers of vertices. 

Presentations of the algebras $P_k(t), R_k(t), PR_k(t), B_k(t), TL_k(t), RB_k(t)$, and $M_k(t)$ can be found in \cite{Hal05}, \cite{Lip96}, \cite{Her06}, \cite{Kud06}, \cite{Kau90}, \cite{Hal14}, and \cite{Ben14} respectively. However, these presentations involve many relations, and the number of relations increases as $k$ increases. A more efficient way to study such algebras is to define strict $\kk$-linear monoidal categories whose endomorphism algebras are precisely these algebras. For instance, the definitions of the partition category, the Brauer category, the Temperley-Lieb category can be found in \cite{Del07}, \cite{Leh15}, \cite{Che14} respectively. However, the rook category, the planar rook category, the rook-Brauer category, and the Motzkin category have not appeared in the literature yet. 

This paper is an attempt to give the definitions of the rook category, the planar rook category, the rook-Brauer category, and the Motzkin category in terms of generators and relations. The organization of this paper is as follows. In \cref{lmcategorysection}, we review strict $\kk$-linear monoidal categories and string diagrams. In \cref{partitioncategorysection}, we recall the definition of the partition category $\Par(t)$, and we give some decompositions of partition diagrams. In \cref{prcategorysection}, we define the planar rook category $\PR(t)$ in terms of generators and relations, and we show that the morphism spaces of the planar rook category $\PR(t)$ have linear bases given by all planar rook diagrams.  Following similar procedures, we discuss the rook category $\cR(t)$, the rook-Brauer category $\RB(t)$, and the Motzkin category $\cM(t)$ in \cref{rookcategorysection,rbcategorysection,motzkincategorysection}.

All the categories we define in this paper are subcategories of the partition category. Recently, in \cite{Sav19}, the partition category was embedded in the Heisenberg category. Therefore, all the categories we define here are also embedded in the Heisenberg category.

\subsection*{Notation}
Throughout, let $\kk$ denotes an arbitrary commutative ring. Let $\N$ denote the additive monoid of nonnegative integers.

\subsection*{Acknowledgements}
This research project was supported by the Mitacs Globalink, the China Scholarship Council, and the University of Ottawa, and supervised by Professor Alistair Savage. The author would like to express her sincere gratitude to Professor Alistair Savage for his patience and guidance throughout the whole project.

\section{Strict $\kk$-linear monoidal categories and string diagrams\label{lmcategorysection}}

\subsection{Strict $\kk$-linear monoidal categories}

We give here a quick review of strict $\kk$-linear monoidal categories. We follow the definitions given in \cite[\S1]{Tur17} and \cite[\S3]{Sav18}.

\begin{defin}\label{strictmonoidalcategory}
  A \emph{strict monoidal category} $(\cC,\otimes,\one)$ is a category $\cC$ equipped with
  \begin{itemize}
    \item a bifunctor (the \emph{tensor product}) $\otimes\colon\cC\times\cC\to\cC$, and
    \item a \emph{unit object} $\one$, 
  \end{itemize}
  such that, for all objects $X,Y$, and $Z$ of $\cC$, we have
  \begin{gather}
    \label{ax1} (X\otimes Y)\otimes Z=X\otimes(Y\otimes Z), \\
    \label{ax2} \one\otimes X=X=X\otimes\one,
  \end{gather}
  and, for all morphisms $f,g,$ and $h$ of $\cC$, we have
  \begin{gather}
    \label{ax3} (f\otimes g)\otimes h=f\otimes(g\otimes h), \\
    \label{ax4} 1_{\one}\otimes f=f=f\otimes1_{\one}.
  \end{gather}
  Here, and throughout the document, $1_X$ denotes the identity endomorphism on object $X$. For convenience, we often denote the strict monoidal category $(\cC,\otimes,\one)$ simply by $\cC$.
\end{defin}

Note that, in a (not necessarily strict) monoidal category, the equalities in \cref{strictmonoidalcategory} are replaced by isomorphisms, and one imposes certain coherence laws. (See \cite[\S1.2.1]{Tur17} for more details.) However, one version of Mac Lane's coherence states that every monoidal category is monoidally equivalent to a strict one. (See \cite[\S10.5]{Kas95} or \cite{Sch01} for a proof.) Therefore, we do not lose much by assuming monoidal categories are strict.

\begin{defin}\label{klinearcategory}
  A \emph{$\kk$-linear category} is a category $\cC$ such that
  \begin{itemize}
    \item for any two objects $X$ and $Y$ of $\cC$, the morphism space $\Hom_{\cC}(X,Y)$ is a $\kk$-module,
    \item composition of morphisms is $\kk$-bilinear:
    \[f\circ(\alpha g+\beta h)=\alpha(f\circ g)+\beta(f\circ h),\]
    \[(\alpha f+\beta g)\circ h=\alpha(f\circ h)+\beta(g\circ h),\]
    for all $\alpha,\beta\in\kk$ and morphisms $f,g,$ and $h$ such that the above operations are defined.
  \end{itemize}
\end{defin}

A \emph{strict $\kk$-linear monoidal category} is a category that is both strict monoidal and $\kk$-linear, and such that the tensor product of morphisms is $\kk$-bilinear:
  \[f\otimes(\alpha g+\beta h)=\alpha(f\otimes g)+\beta(f\otimes h),\]
  \[(\alpha f+\beta g)\otimes h=\alpha(f\otimes h)+\beta(g\otimes h),\]
for all $\alpha,\beta\in\kk$ and morphisms $f,g,$ and $h$. 

\begin{defin}\label{monoidalfunctor}
  Let $(\cC,\otimes,\one_{\cC})$ and $(\cD,\otimes,\one_{\cD})$ be two strict $\kk$-linear monoidal categories. A \emph{strict $\kk$-linear monoidal functor} from $\cC$ to $\cD$ is a functor $F\colon\cC\to\cD$ such that
  \begin{itemize}
    \item $F(X\otimes Y)=F(X)\otimes F(Y),\ \text{for all objects}\ X,Y$,
    \item $F(f\otimes g)=F(f)\otimes F(g),\ \text{for all morphisms}\ f,g$,
    \item $F(\one_\cC)=\one_{\cD}$,
  \end{itemize}
  and its action on the morphism space is $\kk$-linear, that is, for all objects $X,Y$, the map
  \[
    \Hom_{\cC}(X,Y)\to\Hom_{\cD}(F(X),F(Y)),\quad f\mapsto F(f)
  \]
 is $\kk$-linear.
\end{defin}

\subsection{String diagrams}

There is a very intuitive way to represent strict monoidal categories --- \emph{string diagrams}. We give a brief introduction to string diagrams here. Interested readers should refer to \cite[\S2.1]{Tur17} or \cite[\S2.3]{Sav18} for more details.

Given a strict monoidal category $\cC$, the identity endomorphism $\id_X$ of an object $X$ of $\cC$, a morphism $f\colon X\to Y$ in $\cC$, the composition of two morphisms $f\colon X\to Y$ and $g\colon Y\to Z$, the tensor product of two morphisms $f\colon X\to Y$ and $g\colon Z\to W$ can be drawn as the following diagrams:
\[
  \id_X\ =\
  \begin{tikzpicture}[anchorbase]
    \draw(0,0) -- (0,1.2);
    \filldraw[black] (0,1.2);
    \filldraw[black] (0,0) node[anchor=north]{$X$};
  \end{tikzpicture}
  \ , \qquad f\ =\ 
  \begin{tikzpicture}[anchorbase]
    \draw(0,0) -- (0,1.2);
    \filldraw[black] (0,1.2) node[anchor=south]{$Y$};
    \filldraw[black] (0,0) node[anchor=north]{$X$};
    \filldraw[color=black, fill=white] (0,0.6) circle (0.25) node{$f$};
  \end{tikzpicture}
  \ , \qquad f\circ g\ =\ 
  \begin{tikzpicture}[anchorbase]
    \draw(0,0) -- (0,2);
    \filldraw[color=black, fill=white] (0,1.33) circle (0.25) node{$f$};
    \filldraw[color=black, fill=white] (0,0.67) circle (0.25) node{$g$};
    \filldraw[black] (0,2) node[anchor=south]{$Z$};
    \filldraw[black] (0,0) node[anchor=north]{$X$};
  \end{tikzpicture}
  \ , \qquad f\otimes g\ =\ 
  \begin{tikzpicture}[anchorbase]
    \draw(0,0) -- (0,1.2);
    \draw(0.7,0) -- (0.7,1.2); 
    \filldraw[color=black, fill=white] (0,0.6) circle (0.25) node{$f$};
    \filldraw[color=black, fill=white] (0.7,0.6) circle (0.25) node{$g$};
    \filldraw[black] (0,1.2) node[anchor=south]{$Y$};
    \filldraw[black] (0,0) node[anchor=north]{$X$};
    \filldraw[black] (0.7,1.2) node[anchor=south]{$W$};
    \filldraw[black] (0.7,0) node[anchor=north]{$Z$};
  \end{tikzpicture}
\]
Sometimes, we would omit the object labels $X,Y,Z,W,\cdots$ when they are clear or unimportant. Furthermore, we would like to mention the important \emph{interchange law} for strict monoidal categories:
\begin{equation}\label{interchangelaw}
  \begin{tikzpicture}[anchorbase]
    \draw(0,0) -- (0,2.1);
    \draw(0.7,0) -- (0.7,2.1); 
    \filldraw[color=black, fill=white] (0,0.7) circle (0.25) node{$f$};
    \filldraw[color=black, fill=white] (0.7,1.4) circle (0.25) node{$g$};
  \end{tikzpicture}
  \ =\
  \begin{tikzpicture}[anchorbase]
    \draw(0,0) -- (0,2.1);
    \draw(0.7,0) -- (0.7,2.1); 
    \filldraw[color=black, fill=white] (0,1.05) circle (0.25) node{$f$};
    \filldraw[color=black, fill=white] (0.7,1.05) circle (0.25) node{$g$};
  \end{tikzpicture}
  \ =\
  \begin{tikzpicture}[anchorbase]
    \draw(0,0) -- (0,2.1);
    \draw(0.7,0) -- (0.7,2.1); 
    \filldraw[color=black, fill=white] (0,1.4) circle (0.25) node{$f$};
    \filldraw[color=black, fill=white] (0.7,0.7) circle (0.25) node{$g$};
  \end{tikzpicture}\ .
\end{equation}

In a strict monoidal category, we have two kinds of compositions: the traditional composition, which is called the \emph{vertical composition}, and the tensor product, which is called the \emph{horizontal composition}. When drawing string diagrams, the vertical composition $f\circ g$ is obtained by stacking $f$ above $g$, and the horizontal composition $f\otimes g$ is obtained by juxtapositing $g$ on the right of $f$.

\subsection{Presentations}

Just as one can define associative algebras via generators and relations, one can also define strict $\kk$-linear monoidal categories in this way. We follow the definition given in \cite[\S3]{Sav18}.

To define a strict $\kk$-linear monoidal category via a presentation, we should specify a set of generating objects, a set of generating morphisms, and some relations on \emph{morphisms} (not on objects!). If $\{X_i\colon i\in I\}$ is our set of generating objects, then an arbitrary object in our category is a finite tensor product of these generating objects:
$$X_{i_1}\otimes X_{i_2}\otimes\cdots\otimes X_{i_n},\ i_1,i_2,\cdots i_n\in I,\ n\in\N.$$
We think of $\one$ as being the ``empty tensor product''. If $\{f_j\colon j\in J\}$ is our set of generating morphisms, then we can take arbitrary tensor products and compositions (when domains and codomains match) of these generators, e.g.,
$$(f_{j_1}\otimes f_{j_2}\otimes f_{j_3})\circ(f_{j_4}\otimes f_{j_4}),\ j_1,j_2,j_3,j_4,j_5\in J.$$

When we define a strict $\kk$-linear monoidal category $\cC$ via a presentation, we get a very intuitive interpretation of both the vertical and horizontal compositions; however, we do not get much information about the $\kk$-linear structure on $\cC$. For example, we do not know whether the morphism space is finite-dimensional, or whether it has a linear basis without further discussions.

As a concrete example, we introduce the symmetric group category to the readers.

\begin{defin}\label{symmetricgroupcategory} 
  The \emph{symmetric group category} $\cS$ is the strict $\kk$-linear monoidal category with
  \begin{itemize}
    \item one generating object: $|$,
    \item one generating morphism:
    \[
      \begin{tikzpicture}[anchorbase]
        \draw (0,0) to (0.7,0.7);
        \draw (0.7,0) to (0,0.7);
      \end{tikzpicture}
      \colon
      |\otimes |\to |\otimes|
      \ ,
    \]
    \item two relations:
    \[
      \begin{tikzpicture}[anchorbase]
        \draw (0,0) \braidto (0.7,0.7) \braidto (0,1.4);
        \draw (0.7,0) \braidto (0,0.7) \braidto (0.7,1.4);
      \end{tikzpicture}
      \ =\ 
      \begin{tikzpicture}[anchorbase]
        \draw (0,0) to (0,1.4);
        \draw (0.5,0) to (0.5,1.4);
      \end{tikzpicture}
      \quad        \text{and}\quad
      \begin{tikzpicture}[anchorbase]
        \draw (0,0) \braidto (0.7,0.7) \braidto (1.4,1.4) to (1.4,2.1);
        \draw (0.7,0) \braidto (0,0.7) to (0,1.4) \braidto (0.7,2.1);
        \draw (1.4,0) to (1.4,0.7) \braidto (0.7,1.4) \braidto (0,2.1);
      \end{tikzpicture}
      \ =\ 
      \begin{tikzpicture}[anchorbase]
        \draw (0,0) to (0,0.7) \braidto (0.7,1.4) \braidto (1.4,2.1);
        \draw (0.7,0) \braidto (1.4,0.7) to (1.4,1.4) \braidto (0.7,2.1);
        \draw (1.4,0) \braidto (0.7,0.7) \braidto (0,1.4) to (0,2.1);
      \end{tikzpicture}
      \ .
    \]
  \end{itemize} 
\end{defin}

Let $S_n$ denote the symmetric group of order $n$. Then the endomorphism algebra $\End_{\cS}(|^{\otimes n})$ is isomorphic to the group algebra $\kk S_n$. (See \cite[\S6.2]{Liu18} for a precise proof.)

\section{The partition category $\Par(t)$\label{partitioncategorysection}}

We recall the definition of the partition category following \cite[\S2]{Com16} and \cite[\S2]{Sav19}.

\subsection{Partition diagrams\label{partitiondiagramssubsection}}
We have defined partition diagrams in \cref{introduction}. Note that there may be more than one way to draw a partition. For example, the partition $\big\{ \{1,3,1'\} , \{2,4\} , \\\{5,3',5'\} , \{7,2\} , \{6\} , \{4'\} \big\}$ of type $\binom{5}{7}$ can be depicted as
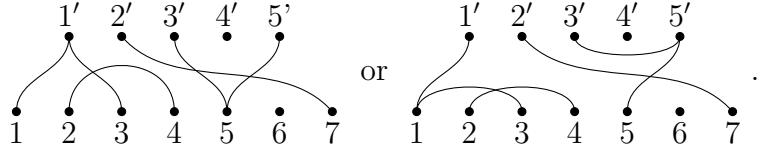
\[
  \begin{tikzpicture}[anchorbase]
    \pd{0.7,1} node[anchor=south] {$1'$};
    \pd{1.4,1} node[anchor=south] {$2'$};
    \pd{2.1,1} node[anchor=south] {$3'$};
    \pd{2.8,1} node[anchor=south] {$4'$};
    \pd{3.5,1} node[anchor=south] {$5$'};  
    \pd{0,0} node[anchor=north] {$1$};
    \pd{0.7,0} node[anchor=north] {$2$};
    \pd{1.4,0} node[anchor=north] {$3$};
    \pd{2.1,0} node[anchor=north] {$4$};
    \pd{2.8,0} node[anchor=north] {$5$};
    \pd{3.5,0} node[anchor=north] {$6$};
    \pd{4.2,0} node[anchor=north] {$7$};
    \draw (0,0) \braidto (0.7,1);
    \draw (1.4,0) \braidto (0.7,1);
    \draw (0.7,0) to[out=up,in=up,looseness=1.5] (2.1,0);
    \draw (2.8,0) \braidto (2.1,1);
    \draw (2.8,0) \braidto (3.5,1);
    \draw (4.2,0) ..controls (3.5,0.8) and (2.1,0.2).. (1.4,1);
  \end{tikzpicture}
  \ \text{or}\ 
  \begin{tikzpicture}[anchorbase]
    \pd{0.7,1} node[anchor=south] {$1'$};
    \pd{1.4,1} node[anchor=south] {$2'$};
    \pd{2.1,1} node[anchor=south] {$3'$};
    \pd{2.8,1} node[anchor=south] {$4'$};
    \pd{3.5,1} node[anchor=south] {$5'$};  
    \pd{0,0} node[anchor=north] {$1$};
    \pd{0.7,0} node[anchor=north] {$2$};
    \pd{1.4,0} node[anchor=north] {$3$};
    \pd{2.1,0} node[anchor=north] {$4$};
    \pd{2.8,0} node[anchor=north] {$5$};
    \pd{3.5,0} node[anchor=north] {$6$};
    \pd{4.2,0} node[anchor=north] {$7$};
    \draw (0,0) \braidto (0.7,1);
    \draw (0,0) to[out=up,in=up,looseness=0.8] (1.4,0);
    \draw (0.7,0) to[out=up,in=up,looseness=0.8] (2.1,0);
    \draw (2.1,1) to[out=down,in=down,looseness=0.6] (3.5,1);
    \draw (2.8,0) \braidto (3.5,1);
    \draw (4.2,0) ..controls (3.5,0.8) and (2.1,0.2).. (1.4,1);
  \end{tikzpicture}.
\]
However, as long as two partition diagrams give rise to the same set partition, we consider them to be equivalent. We may omit the labels of the vertices when drawing partition diagrams if no confusion will be caused. We write $D\colon k\to l$ to indicate that $D$ is a partition of type $\binom{l}{k}$.

A block of size $1$ is called a \emph{singleton}. A vertex in a singleton is called an \emph{isolated vertex}. Isolated vertices will play a vital role in our discussions of diagram categories. There is by convention just one partition of type $\binom{0}{0}$, and we denote the unique partition diagrams of types $\binom{1}{0}$ and $\binom{0}{1}$ by
\[
  \begin{tikzpicture}[anchorbase]
    \pd{0,0};
    \draw (0,0) to (0,-0.25);
  \end{tikzpicture}
  \colon 0 \to 1
  \qquad \text{and} \qquad
  \begin{tikzpicture}[anchorbase]
    \pd{0,0};
    \draw (0,0) to (0,0.25);
  \end{tikzpicture}
  \colon 1 \to 0.
\]

Given two partitions $D'\colon m\to k$, $D\colon k\to l$, one can stack $D$ on top of $D'$ to obtain a diagram $\begin{matrix}D \\ D'\end{matrix}$ with three rows of vertices. Let $\alpha(D,D')$ denote the number of connected components in $\begin{matrix}D \\ D'\end{matrix}$ all of whose vertices are in the middle row. Let $D\star D'$ denote a partition diagram of type $\binom{l}{m}$ with the following property: vertices are in the same block of $D\star D'$ if and only if the corresponding vertices in the top and bottom rows of $\begin{matrix}D\\ D'\end{matrix}$ are in the same block.

\subsection{The partition category $\Par(t)$\label{partitioncategorysubsection}}

Fix $t\in\kk$. The \emph{partition category} $\Par(t)$ is the strict $\kk$-linear monoidal category whose objects are nonnegative integers and, given two objects $k,l$ of $\Par(t)$, the morphism space $\Hom_{\Par(t)}(k,l)$ consists of all formal $\kk$-linear combinations of partitions of type $\binom{l}{k}$. The vertical composition $\circ$ is given by 
\[
  D\circ D'=t^{\alpha(D,D')}D\star D'
\]
for composable partition diagrams $D$ and $D'$, and extended by linearity. The tensor product $\otimes$ is given as follows.
\begin{itemize}
  \item on objects: $k\otimes l=k+l,\ k,l\in\N,$ with a unit object $0$,
  \item on morphisms: $D\otimes D'$ is obtained by horizontally juxtapositing $D$ to the right of $D'$.
\end{itemize}
For example, if 
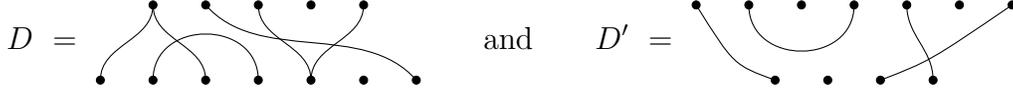
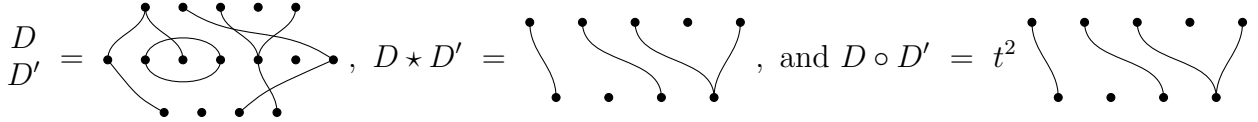
\[
  D\ =\
  \begin{tikzpicture}[anchorbase]
    \pd{0.7,1};\pd{1.4,1};\pd{2.1,1};\pd{2.8,1};\pd{3.5,1};  
    \pd{0,0};\pd{0.7,0};\pd{1.4,0};\pd{2.1,0};\pd{2.8,0};\pd{3.5,0};\pd{4.2,0};
    \draw (0,0) \braidto (0.7,1);
    \draw (1.4,0) \braidto (0.7,1);
    \draw (0.7,0) to[out=up,in=up,looseness=1.5] (2.1,0);
    \draw (2.8,0) \braidto (2.1,1);
    \draw (2.8,0) \braidto (3.5,1);
    \draw (4.2,0) ..controls (3.5,0.8) and (2.1,0.2).. (1.4,1);
  \end{tikzpicture}
  \qquad \text{and} \qquad D'\ =\ 
  \begin{tikzpicture}[anchorbase]
    \pd{0,1};\pd{0.7,1};\pd{1.4,1};\pd{2.1,1};\pd{2.8,1};\pd{3.5,1};\pd{4.2,1};
    \pd{1.05,0};\pd{1.75,0};\pd{2.45,0};\pd{3.15,0};
    \draw (1.05,0) ..controls (0.5,0.2).. (0,1);
    \draw (0.7,1) to[out=down,in=down,looseness=1.5] (2.1,1);
    \draw (2.45,0) ..controls (3.1,0.25).. (4.2,1);
    \draw (3.15,0) \braidto (2.8,1);
  \end{tikzpicture}
\]
then
\[
  \begin{matrix}D\\D'\end{matrix}\ =\ 
  \begin{tikzpicture}[anchorbase]
    \pd{0.5,0.7};\pd{1,0.7};\pd{1.5,0.7};\pd{2,0.7};\pd{2.5,0.7};
    \pd{0,0};\pd{0.5,0};\pd{1,0};\pd{1.5,0};\pd{2,0};\pd{2.5,0};\pd{3,0};
    \pd{0.75,-0.7};\pd{1.25,-0.7};\pd{1.75,-0.7};\pd{2.25,-0.7};
    \draw (0,0) \braidto (0.5,0.7);
    \draw (1,0) \braidto (0.5,0.7);
    \draw (0.5,0) to[out=up,in=up,looseness=1] (1.5,0);
    \draw (2,0) \braidto (1.5,0.7);
    \draw (2,0) \braidto (2.5,0.7);
    \draw (3,0) ..controls (2.3,0.5) and (1.7,0.2).. (1,0.7);
    \draw (0.75,-0.7) ..controls (0.4,-0.5).. (0,0);
    \draw (0.5,0) to[out=down,in=down,looseness=1] (1.5,0);
    \draw (1.75,-0.7) ..controls (2.1,-0.4).. (3,0);
    \draw (2.25,-0.7) \braidto (2,0); 
  \end{tikzpicture}
  \ ,\ 
  D\star D'\ =\
  \begin{tikzpicture}[anchorbase]
    \pd{0.7,0};\pd{1.4,0};\pd{2.1,0};\pd{2.8,0};\pd{3.5,0};  
    \pd{1.05,-1};\pd{1.75,-1};\pd{2.45,-1};\pd{3.15,-1};
    \draw (1.05,-1) \braidto (0.7,0);
    \draw (2.45,-1) \braidto (1.4,0);
    \draw (3.15,-1) \braidto (2.1,0);
    \draw (3.15,-1) \braidto (3.5,0);
  \end{tikzpicture}
  \ ,\ \text{and}\ 
  D\circ D'\ =\ t^2\ 
  \begin{tikzpicture}[anchorbase]
    \pd{0.7,0};\pd{1.4,0};\pd{2.1,0};\pd{2.8,0};\pd{3.5,0};   
    \pd{1.05,-1};\pd{1.75,-1};\pd{2.45,-1};\pd{3.15,-1};
    \draw (1.05,-1) \braidto (0.7,0);
    \draw (2.45,-1) \braidto (1.4,0);
    \draw (3.15,-1) \braidto (2.1,0);
    \draw (3.15,-1) \braidto (3.5,0);
  \end{tikzpicture}
\]

Fix $k\in\N$. The endomorphism algebra $\End_{\Par(t)}(k)$ is the \emph{partition algebra} $P_k(t)$ given in \cref{introduction}.

There are two opposite categories of the partition category $\Par(t)$, denoted by $\Par(t)^{\op}$, which reverses the vertical composition, and $\Par(t)^{\otimes\op}$, which reverses the horizontal composition; see \cite[\S1.2.2]{Tur17} for more details. The reflection of a partition diagram in a horizontal or vertical line is still a partition diagram. We can use this symmetry to define two involutions from the partition category $\Par(t)$ to its opposite categories, which will simplify definitions and proofs a lot. The first involution $^*\colon \Par(t)\to\Par(t)^{\op}$ takes each object to itself and takes each diagram to its reflection in a horizontal line, and the second involution $^{\sharp}\colon \Par(t)\to\Par(t)^{\otimes\op}$ takes each object to itself and takes each diagram to its reflection in a vertical line.\color{black}

For example, if
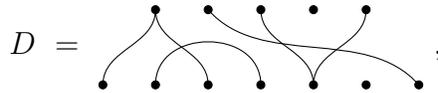
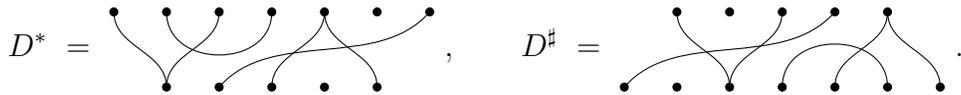
\[
  D\ =\
  \begin{tikzpicture}[anchorbase]
    \pd{0.7,1};\pd{1.4,1};\pd{2.1,1};\pd{2.8,1};\pd{3.5,1};  
    \pd{0,0};\pd{0.7,0};\pd{1.4,0};\pd{2.1,0};\pd{2.8,0};\pd{3.5,0};\pd{4.2,0};
    \draw (0,0) \braidto (0.7,1);
    \draw (1.4,0) \braidto (0.7,1);
    \draw (0.7,0) to[out=up,in=up,looseness=1.4] (2.1,0);
    \draw (2.8,0) \braidto (2.1,1);
    \draw (2.8,0) \braidto (3.5,1);
    \draw (4.2,0) ..controls (3.5,0.8) and (2.1,0.2).. (1.4,1);
  \end{tikzpicture}\ ,
\]
then 
\[
  D^*\ =\ 
  \begin{tikzpicture}[anchorbase]
    \pd{0,1};\pd{0.7,1};\pd{1.4,1};\pd{2.1,1};\pd{2.8,1};\pd{3.5,1};\pd{4.2,1};
    \pd{0.7,0};\pd{1.4,0};\pd{2.1,0};\pd{2.8,0};\pd{3.5,0};
    \draw (0.7,0) \braidto (0,1);
    \draw (0.7,0) \braidto (1.4,1);
    \draw (2.1,0) \braidto (2.8,1);
    \draw (3.5,0) \braidto (2.8,1);
    \draw (1.4,0) ..controls (2.1,0.8) and (3.5,0.2).. (4.2,1);
    \draw (0.7,1) to[out=down,in=down,looseness=1.4] (2.1,1);
  \end{tikzpicture}
  \ ,\qquad D^{\sharp}\ =\ 
  \begin{tikzpicture}[anchorbase]
    \pd{0.7,1};\pd{1.4,1};\pd{2.1,1};\pd{2.8,1};\pd{3.5,1};  
    \pd{0,0};\pd{0.7,0};\pd{1.4,0};\pd{2.1,0};\pd{2.8,0};\pd{3.5,0};\pd{4.2,0};
    \draw (1.4,0) \braidto (0.7,1);
    \draw (1.4,0) \braidto (2.1,1);
    \draw (2.8,0) \braidto (3.5,1);
    \draw (4.2,0) \braidto (3.5,1);
    \draw (0,0) ..controls (0.7,0.8) and (2.1,0.2).. (2.8,1);
    \draw (2.1,0) to[out=up,in=up,looseness=1.4] (3.5,0); 
  \end{tikzpicture}\ .
\]
It is easy to check that $(D^*)^*=D$, $(D_1\circ D_2)^*=D_2^*\circ D_1^*$, $(D_1\otimes D_2)^*=(D_1)^*\otimes (D_2)^*$, and that $(D^{\sharp})^{\sharp}=D$, $(D_1\circ D_2)^{\sharp}=D_1^{\sharp}\circ D_2^{\sharp}$, $(D_1\otimes D_2)^{\sharp}=D_2^{\sharp}\otimes D_1^{\sharp}$. Moreover, It is useful to point out that the particular sets of diagrams defined in \cref{introduction} are left invariant by the involutions $^*$ and $^{\sharp}$.

\begin{defin}\label{definskeleton}
Let $D$ be a partition of type $\binom{l}{k}$. We define the \emph{skeleton} of $D$ to be the partition obtained by removing all the singletons from $D$, and we denote it by $\widetilde{D}$. 
\end{defin}
 
For example, the skeleton of the partition diagram
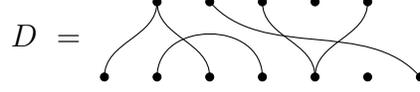
\[
  D\ =\
  \begin{tikzpicture}[anchorbase]
    \pd{0.7,1};\pd{1.4,1};\pd{2.1,1};\pd{2.8,1};\pd{3.5,1};  
    \pd{0,0};\pd{0.7,0};\pd{1.4,0};\pd{2.1,0};\pd{2.8,0};\pd{3.5,0};\pd{4.2,0};
    \draw (0,0) \braidto (0.7,1);
    \draw (1.4,0) \braidto (0.7,1);
    \draw (0.7,0) to[out=up,in=up,looseness=1.4] (2.1,0);
    \draw (2.8,0) \braidto (2.1,1);
    \draw (2.8,0) \braidto (3.5,1);
    \draw (4.2,0) ..controls (3.5,0.8) and (2.1,0.2)..  (1.4,1);
  \end{tikzpicture}
\]
is
\[
  \widetilde{D}\ =\ 
  
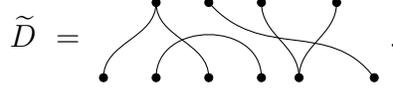
\begin{tikzpicture}[anchorbase]
    \pd{0.7,1};\pd{1.4,1};\pd{2.1,1};\pd{3.1,1};  
    \pd{0,0};\pd{0.7,0};\pd{1.4,0};\pd{2.1,0};\pd{2.6,0};\pd{3.6,0};
    \draw (0,0) \braidto (0.7,1);
    \draw (1.4,0) \braidto (0.7,1);
    \draw (0.7,0) to[out=up,in=up,looseness=1.4] (2.1,0);
    \draw (2.6,0) \braidto (2.1,1);
    \draw (2.6,0) \braidto (3.1,1);
    \draw (3.6,0) ..controls (3.0,0.8) and (2.1,0.2)..  (1.4,1);
  \end{tikzpicture}
  \ .
\]

\begin{prop}\label{partitiondecompose}
  Every partition diagram $D$ has a decomposition $D=P_1\circ\widetilde{D}\circ P_2$ where $\widetilde{D}$ is the skeleton of $D$ and $P_1,P_2$ are planar rook diagrams. 
\end{prop}

Instead of writing out a formal proof, we give an example to illustrate how to achieve this. Consider the following partition diagram.
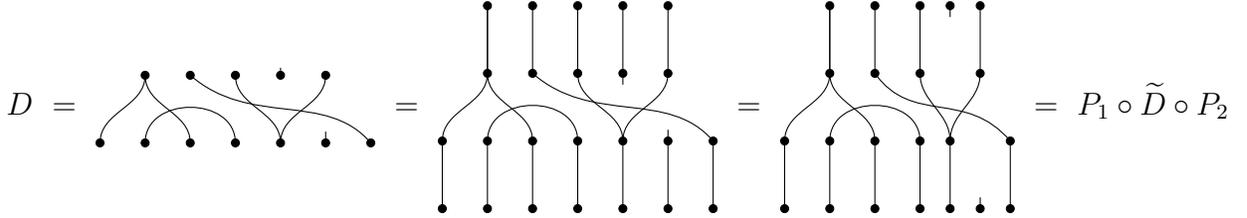
\[
  D\ =\
  \begin{tikzpicture}[anchorbase]
    \pd{0.6,0.9};\pd{1.2,0.9};\pd{1.8,0.9};\pd{2.4,0.9};\pd{3,0.9};  
    \pd{0,0};\pd{0.6,0};\pd{1.2,0};\pd{1.8,0};\pd{2.4,0};\pd{3,0};\pd{3.6,0};
    \draw (0,0) \braidto (0.6,0.9);
    \draw (1.2,0) \braidto (0.6,0.9);
    \draw (0.6,0) to[out=up,in=up,looseness=1.35] (1.8,0);
    \draw (2.4,0) \braidto (1.8,0.9);
    \draw (2.4,0) \braidto (3,0.9);
    \draw (3.6,0) ..controls (3,0.7) and (2,0.2).. (1.2,0.9);
    \draw (3,0) to (3,0.15);
    \draw (2.4,0.85) to (2.4,1);
  \end{tikzpicture}
  \ =\ 
  \begin{tikzpicture}[anchorbase]
    \pd{0.6,1.8};\pd{1.2,1.8};\pd{1.8,1.8};\pd{2.4,1.8};\pd{3,1.8};
    \pd{0.6,0.9};\pd{1.2,0.9};\pd{1.8,0.9};\pd{2.4,0.9};\pd{3,0.9};  
    \pd{0,0};\pd{0.6,0};\pd{1.2,0};\pd{1.8,0};\pd{2.4,0};\pd{3,0};\pd{3.6,0};
    \pd{0,-0.9};\pd{0.6,-0.9};\pd{1.2,-0.9};\pd{1.8,-0.9};\pd{2.4,-0.9};\pd{3,-0.9};\pd{3.6,-0.9};
    \draw (0,-0.9) to (0,0) \braidto (0.6,0.9) to (0.6,1.8);
    \draw (1.2,-0.9) to (1.2,0) \braidto (0.6,0.9) to (0.6,1.8);
    \draw (0.6,-0.9) to (0.6,0) to[out=up,in=up,looseness=1.35] (1.8,0) to (1.8,-0.9);
    \draw (2.4,-0.9) to (2.4,0) \braidto (1.8,0.9) to (1.8,1.8);
    \draw (2.4,0) \braidto (3,0.9) to (3,1.8);
    \draw (3.6,-0.9) to (3.6,0) ..controls (3,0.7) and (2,0.2).. (1.2,0.9) to (1.2,1.8);
    \draw (3,-0.9) to (3,0) to (3,0.15);
    \draw (2.4,0.75) to (2.4,0.9) to (2.4,1.8);
  \end{tikzpicture}
  \ =\ 
  \begin{tikzpicture}[anchorbase]
    \pd{0.6,1.8};\pd{1.2,1.8};\pd{1.8,1.8};\pd{2.2,1.8};\pd{2.6,1.8};
    \pd{0.6,0.9};\pd{1.2,0.9};\pd{1.8,0.9};\pd{2.6,0.9};
    \pd{0,0};\pd{0.6,0};\pd{1.2,0};\pd{1.8,0};\pd{2.2,0};\pd{3,0};
    \pd{0,-0.9};\pd{0.6,-0.9};\pd{1.2,-0.9};\pd{1.8,-0.9};\pd{2.2,-0.9};\pd{2.6,-0.9};\pd{3,-0.9};
    \draw (0,-0.9) to (0,0) \braidto (0.6,0.9) to (0.6,1.8);
    \draw (1.2,-0.9) to (1.2,0) \braidto (0.6,0.9) to (0.6,1.8);
    \draw (0.6,-0.9) to (0.6,0) to[out=up,in=up,looseness=1.35] (1.8,0) to (1.8,-0.9);
    \draw (2.2,-0.9) to (2.2,0) \braidto (1.8,0.9) to (1.8,1.8);
    \draw (2.2,0) \braidto (2.6,0.9) to (2.6,1.8);
    \draw (3,-0.9) to (3,0) ..controls (2.4,0.7) and (1.7,0.2).. (1.2,0.9) to (1.2,1.8);
    \draw (2.6,-0.9) to (2.6,-0.75);
    \draw (2.2,1.65) to (2.2,1.8);
  \end{tikzpicture}
  \ =\ 
  P_1\circ\widetilde{D}\circ P_2
\]
Note that this decomposition only uses the axioms of strict $\kk$-linear monoidal categories. Moreover, we have the following correspondences between diagrams and their skeletons.
\begin{center}
\begin{tabular}{c|c}
  \hline
  \hline
   \textbf{Diagrams} & \textbf{Skeletons} \\
  \hline
  \hline
   rook diagram & permutation diagram \\
  \hline
   rook-Brauer diagram & Brauer diagram \\
  \hline
   Motzkin diagram & Temperley-Lieb diagram \\
  \hline
  \hline
\end{tabular}
\end{center}

Then, as a corollary of \cref{partitiondecompose}, we have
\begin{cor}\label{partitiondecomposecor}
  \begin{enumerate}
    \renewcommand{\theenumi}{(\alph{enumi})}
    \item\label{decompose1a} Every rook diagram $D$ has a decomposition $D=P_1\circ S\circ P_2$ where $S$ is a permutation diagram and $P_1,P_2$ are planar rook diagrams.
    \item\label{decompose1b} Every rook-Brauer diagram $D$ has a decomposition $D=P_1\circ B\circ P_2$ where $B$ is a Brauer diagram and $P_1,P_2$ are planar rook diagrams.
    \item\label{decompose1c} Every Motzkin diagram $D$ has a decomposition $D=P_1\circ T\circ P_2$ where $T$ is a Temperley-Lieb diagram and $P_1,P_2$ are planar rook diagrams.
  \end{enumerate}
\end{cor}
Apart from the above decompositions given in \cref{partitiondecomposecor}, rook diagrams and rook-Brauer diagrams have other decompositions given in the following proposition.
\begin{prop}\label{decompose2}
  \begin{enumerate}
    \renewcommand{\theenumi}{(\alph{enumi})}
    \item\label{decompose2a} Every rook diagram $D$ has a decomposition $D=S\circ P$ where $S$ is a permutation diagram, and $P$ is a planar rook diagram.
    \item\label{decompose2b} Every rook-Brauer diagram $D$ has a decomposition $D=B\circ P$ where $B$ is a Brauer diagram, and $P$ is a planar rook diagram.
    \item\label{decompose2c} Every rook-Brauer diagram $D$ has a decomposition $D=S\circ M$ where $S$ is a permutation diagram and $M$ is a Motzkin diagram. 
  \end{enumerate}
\end{prop}

Instead of writing out formal proofs, we will give examples to illustrate how to obtain such decompositions. Consider the following rook diagram.
\[
  D\ =\ 
  \begin{tikzpicture}[anchorbase]
    \pd{0.25,0.7};\pd{0.75,0.7};\pd{1.25,0.7};\pd{1.75,0.7};
    \pd{0,0};\pd{0.5,0};\pd{1,0};\pd{1.5,0};\pd{2,0};
    \draw (0,0) ..controls (0.5,0.4) and (1,0.15).. (1.75,0.7);
    \draw (0.5,0) to (0.5,0.15);
    \draw (1,0) to (0.75,0.7);
    \draw (1.5,0) ..controls (1.1,0.5) and (0.5,0.4).. (0.25,0.7);
    \draw (2,0) to (2,0.15);
    \draw (1.25,0.55) to (1.25,0.7);
  \end{tikzpicture}
  \ =\ 
  \begin{tikzpicture}[anchorbase]
    \pd{0.25,0.7};\pd{0.75,0.7};\pd{1.25,0.7};\pd{1.75,0.7};
    \pd{0,0};\pd{0.5,0};\pd{1,0};\pd{1.5,0};\pd{2,0};
    \pd{0,-0.7};\pd{0.5,-0.7};\pd{1,-0.7};\pd{1.5,-0.7};\pd{2,-0.7};
    \draw (0,-0.7) to (0,0) ..controls (0.5,0.4) and (1,0.15).. (1.75,0.7);
    \draw (0.5,-0.7) to (0.5,0) to (0.5,0.15);
    \draw (1,-0.7) to (1,0) to (0.75,0.7);
    \draw (1.5,-0.7) to (1.5,0) ..controls (1.1,0.5) and (0.5,0.4).. (0.25,0.7);
    \draw (2,-0.7) to (2,0) to (2,0.15);
    \draw (1.25,0.55) to (1.25,0.7);
  \end{tikzpicture}
  \ =\ 
  \begin{tikzpicture}[anchorbase]
    \pd{0.25,0.7};\pd{0.75,0.7};\pd{1.25,0.7};\pd{1.75,0.7};
    \pd{0,0};\pd{1,0};\pd{1.5,0};
    \pd{0,-0.7};\pd{0.5,-0.7};\pd{1,-0.7};\pd{1.5,-0.7};\pd{2,-0.7};
    \draw (0,-0.7) to (0,0) ..controls (0.5,0.4) and (1,0.15).. (1.75,0.7);
    \draw (0.5,-0.7) to (0.5,-0.55);
    \draw (1,-0.7) to (1,0) to (0.75,0.7);
    \draw (1.5,-0.7) to (1.5,0) ..controls (1.1,0.5) and (0.5,0.4).. (0.25,0.7);
    \draw (2,-0.7) to (2,-0.55);
    \draw (1.25,0.55) to (1.25,0.7);
  \end{tikzpicture}
  \ =\ 
  \begin{tikzpicture}[anchorbase]
    \pd{0.25,0.7};\pd{0.75,0.7};\pd{1.25,0.7};\pd{1.75,0.7};
    \pd{0,0};\pd{1,0};\pd{1.5,0};
    \filldraw[green] (1.25,0) circle (1.5pt);
    \pd{0,-0.7};\pd{0.5,-0.7};\pd{1,-0.7};\pd{1.5,-0.7};\pd{2,-0.7};
    \draw (0,-0.7) to (0,0) ..controls (0.5,0.4) and (1,0.15).. (1.75,0.7);
    \draw (0.5,-0.7) to (0.5,-0.55);
    \draw (1,-0.7) to (1,0) to (0.75,0.7);
    \draw (1.5,-0.7) to (1.5,0) ..controls (1.1,0.5) and (0.5,0.4).. (0.25,0.7);
    \draw (2,-0.7) to (2,-0.55);
    \draw[green,-] (1.25,-0.15) to (1.25,0) to (1.25,0.7);
  \end{tikzpicture}
  \ =\ 
  S\circ P
\]
Consider the following rook-Brauer diagram.
\[
  D\ =\ 
  \begin{tikzpicture}[anchorbase]
    \pd{0.75,0.7};\pd{1.25,0.7};\pd{1.75,0.7};\pd{2.25,0.7};\pd{2.75,0.7};
    \pd{0,0};\pd{0.5,0};\pd{1,0};\pd{1.5,0};\pd{2,0};\pd{2.5,0};\pd{3,0};\pd{3.5,0};
    \draw (0,0) to (0,0.15);
    \draw (0.5,0) ..controls (1.3,0.5) and (2.3,0.2).. (2.75,0.7);
    \draw (1,0) to[out=up,in=up,looseness=0.7] (2,0);
    \draw (1.5,0) \braidto (1.75,0.7);
    \draw (2.5,0) to (2.5,0.15);
    \draw (3,0) to[out=up,in=up,looseness=1.5] (3.5,0);
    \draw (0.75,0.7) to[out=down,in=down,looseness=1] (1.25,0.7);
    \draw (2.25,0.55) to (2.25,0.7);
  \end{tikzpicture}
  \ =\ 
  \begin{tikzpicture}[anchorbase]
    \pd{0.75,0.7};\pd{1.25,0.7};\pd{1.75,0.7};\pd{2.25,0.7};\pd{2.75,0.7};
    \pd{0,0};\pd{0.5,0};\pd{1,0};\pd{1.5,0};\pd{2,0};\pd{2.5,0};\pd{3,0};\pd{3.5,0};
    \pd{0,-0.7};\pd{0.5,-0.7};\pd{1,-0.7};\pd{1.5,-0.7};\pd{2,-0.7};\pd{2.5,-0.7};\pd{3,-0.7};\pd{3.5,-0.7};
    \draw (0,-0.7) to (0,0) to (0,0.15);
    \draw (0.5,-0.7) to (0.5,0) ..controls (1.3,0.5) and (2.3,0.2).. (2.75,0.7);
    \draw (1,-0.7) to (1,0) to[out=up,in=up,looseness=0.7] (2,0) to (2,-0.7);
    \draw (1.5,-0.7) to (1.5,0) \braidto (1.75,0.7);
    \draw (2.5,-0.7) to (2.5,0) to (2.5,0.15);
    \draw (3,-0.7) to (3,0) to[out=up,in=up,looseness=1.5] (3.5,0) to (3.5,-0.7);
    \draw (0.75,0.7) to[out=down,in=down,looseness=1] (1.25,0.7);
    \draw (2.25,0.55) to (2.25,0.7);
  \end{tikzpicture}
  \ =\ 
  \begin{tikzpicture}[anchorbase]
    \pd{0.75,0.7};\pd{1.25,0.7};\pd{1.75,0.7};\pd{2.25,0.7};\pd{2.75,0.7};
    \pd{0.5,0};\pd{1,0};\pd{1.5,0};\pd{2,0};\pd{2.5,0};\pd{3,0};
    \pd{0.25,-0.7};\pd{0.5,-0.7};\pd{1,-0.7};\pd{1.5,-0.7};\pd{2,-0.7};\pd{2.25,-0.7};\pd{2.5,-0.7};\pd{3,-0.7};
    \filldraw[green] (2.25,0) circle (1.5pt);
    \draw (0.25,-0.7) to (0.25,-0.55);
    \draw (0.5,-0.7) to (0.5,0) ..controls (1.3,0.5) and (2.3,0.2).. (2.75,0.7);
    \draw (1,-0.7) to (1,0) to[out=up,in=up,looseness=0.7] (2,0) to (2,-0.7);
    \draw (1.5,-0.7) to (1.5,0) \braidto (1.75,0.7);
    \draw (2.25,-0.7) to (2.25,-0.55);
    \draw (2.5,-0.7) to (2.5,0) to[out=up,in=up,looseness=1.5] (3,0) to (3,-0.7);
    \draw (0.75,0.7) to[out=down,in=down,looseness=1] (1.25,0.7);
    \draw[green,-] (2.25,-0.15) to (2.25,0) to (2.25,0.7);
  \end{tikzpicture}
  \ =\ 
  B\circ P
\]
Consider the following rook-Brauer diagram.
\[
  D\ =\ 
  \begin{tikzpicture}[anchorbase]
    \pd{0.75,0.7};\pd{1.25,0.7};\pd{1.75,0.7};\pd{2.25,0.7};\pd{2.75,0.7};
    \pd{0,0};\pd{0.5,0};\pd{1,0};\pd{1.5,0};\pd{2,0};\pd{2.5,0};\pd{3,0};\pd{3.5,0};
    \draw (0,0) to (0,0.15);
    \draw (0.5,0) ..controls (1.3,0.5) and (2.3,0.2).. (2.75,0.7);
    \draw (1,0) to[out=up,in=up,looseness=0.7] (2,0);
    \draw (1.5,0) \braidto (1.75,0.7);
    \draw (2.5,0) to (2.5,0.15);
    \draw (3,0) to[out=up,in=up,looseness=1.5] (3.5,0);
    \draw (0.75,0.7) to[out=down,in=down,looseness=1] (1.25,0.7);
    \draw (2.25,0.55) to (2.25,0.7);
  \end{tikzpicture}
  \ =\ 
  \begin{tikzpicture}[anchorbase]
    \pd{0.75,0.7};\pd{1.25,0.7};\pd{1.75,0.7};\pd{2.25,0.7};\pd{2.75,0.7};
    \pd{0,0};\pd{0.5,0};\pd{1,0};\pd{1.5,0};\pd{2,0};\pd{2.5,0};\pd{3,0};\pd{3.5,0};
    \pd{0,-0.7};\pd{0.5,-0.7};\pd{1,-0.7};\pd{1.5,-0.7};\pd{2,-0.7};\pd{2.5,-0.7};\pd{3,-0.7};\pd{3.5,-0.7};
    \draw (0,-0.7) to (0,0) to (0,0.15);
    \draw (0.5,-0.7) to (0.5,0) ..controls (1.3,0.5) and (2.3,0.2).. (2.75,0.7);
    \draw (1,-0.7) to (1,0) to[out=up,in=up,looseness=0.7] (2,0) to (2,-0.7);
    \draw (1.5,-0.7) to (1.5,0) \braidto (1.75,0.7);
    \draw (2.5,-0.7) to (2.5,0) to (2.5,0.15);
    \draw (3,-0.7) to (3,0) to[out=up,in=up,looseness=1.5] (3.5,0) to (3.5,-0.7);
    \draw (0.75,0.7) to[out=down,in=down,looseness=1] (1.25,0.7);
    \draw (2.25,0.55) to (2.25,0.7);
  \end{tikzpicture}
  \ =\ 
  \begin{tikzpicture}[anchorbase]
    \pd{0.75,0.7};\pd{1.25,0.7};\pd{1.75,0.7};\pd{2.25,0.7};\pd{2.75,0.7};
    \pd{0,0};\pd{0.5,0};
    \filldraw[green] (1,0) circle (1.5pt);
    \filldraw[green] (1.5,0) circle (1.5pt);
    \pd{2,0};\pd{2.5,0};
    \filldraw[green] (3,0) circle (1.5pt);
    \filldraw[green] (3.5,0) circle (1.5pt);
    \pd{0,-0.7};\pd{0.5,-0.7};\pd{1,-0.7};\pd{1.5,-0.7};\pd{2,-0.7};\pd{2.5,-0.7};\pd{3,-0.7};\pd{3.5,-0.7};
    \draw (0,-0.7) to (0,0) to (0,0.15);
    \draw (0.5,-0.7) \braidto (2,0) \braidto (2.75,0.7);
    \draw[green,-] (1,-0.7) to (1,0) to[out=up,in=up,looseness=0.7] (1.5,0) \braidto (2,-0.7);
    \draw (1.5,-0.7) ..controls (0.5,-0.5).. (0.5,0) \braidto (1.75,0.7);
    \draw (2.5,-0.7) to (2.5,0) to (2.5,0.15);
    \draw[green,-] (3,-0.7) to (3,0) to[out=up,in=up,looseness=1.5] (3.5,0) to (3.5,-0.7);
    \draw (0.75,0.7) to[out=down,in=down,looseness=1] (1.25,0.7);
    \draw (2.25,0.55) to (2.25,0.7);
  \end{tikzpicture}
  \ =\ 
  M\circ S
\]
These decompositions are not unique because we can drag the green vertices or the green caps to any other places in the middle rows.

\begin{rem}
  \begin{enumerate}
    \item In \cref{decompose2}, the decompositions \cref{decompose2a} and \cref{decompose2b} only use the axioms of the strict $\kk$-linear monoidal category and the fact that $\begin{tikzpicture}[anchorbase]
      \pd{0,0};
      \draw (0,0) to (0,-0.25);
    \end{tikzpicture}$ can pass through the braiding $\begin{tikzpicture}[anchorbase]
      \pd{0,0.5};\pd{0.5,0.5};
      \pd{0,0};\pd{0.5,0};
      \draw (0,0) \braidto (0.5,0.5);
      \draw (0.5,0) \braidto (0,0.5);
    \end{tikzpicture}$. The decomposition \cref{decompose2c} also uses the fact that the cap map $\begin{tikzpicture}[anchorbase]
      \pd{0,0};\pd{0.5,0};
      \draw (0,0) to[out=up,in=up,looseness=1.5] (0.5,0);
    \end{tikzpicture}$ can pass through the braiding. That is why these decompositions also hold in the diagram categories we define later, which contain fewer relations.
    \item Another proof for \cref{decompose2}\cref{decompose2a} will be given in the end of \cref{rookoriginsubsection}.
  \end{enumerate}
\end{rem}

\begin{cor}\label{decompose2cor}
  \begin{enumerate}
    \renewcommand{\theenumi}{(\alph{enumi})}
    \item\label{decompose2cora} Every rook diagram $D$ has a decomposition $D=P\circ S$ where $S$ is a permutation diagram, and $P$ is a planar rook diagram.
    \item\label{decompose2corb} Every rook-Brauer diagram $D$ has a decomposition $D=P\circ B$ where $B$ is a Brauer diagram, and $P$ is a planar rook diagram.
    \item\label{decompose2corc} Every rook-Brauer diagram $D$ has a decomposition $D=M\circ S$ where $M$ is a permutation diagram, and $S$ is a Motzkin diagram.
  \end{enumerate}
\end{cor}
\begin{proof}
  We only prove \cref{decompose2cora} here. The proof of the other two is similar.
  Consider a rook diagram $D$. Then the diagram $D^*$ is still a rook diagram. By \cref{decompose2}\cref{decompose2a}, the rook diagram $D^*$ has a decomposition $D^*=S\circ P$, where $S$ is a permutation diagram, and $P$ is a planar rook diagram. Then by taking the involution again, we have 
  \begin{equation*}
    D=(D^*)^*=(S\circ P)^*=P^*\circ S^*, 
  \end{equation*}
  where $S^*$ is a permutation diagram, and $P^*$ is a planar rook diagram.
\end{proof}
\begin{theo}\label{partitioncategorypresentation}
  As a strict $\kk$-linear monoidal category, the \emph{partition category} $\Par(t)$ is generated by the object $1$ and the morphisms
  \[
    \mu\ =\
    \begin{tikzpicture}[anchorbase]
      \pd{0.35,0.7};
      \pd{0,0};\pd{0.7,0};
      \draw (0,0) \braidto (0.35,0.7);
      \draw (0.7,0) \braidto (0.35,0.7);
    \end{tikzpicture}
    \colon 2 \to 1,\ 
    \delta\ =\ 
    \begin{tikzpicture}[anchorbase]
      \pd{0,0.7};\pd{0.7,0.7};
      \pd{0.35,0};
      \draw (0.35,0) \braidto (0,0.7);
      \draw (0.35,0) \braidto (0.7,0.7);
    \end{tikzpicture}
    \colon 1 \to 2,\ 
    s\ =\
    \begin{tikzpicture}[anchorbase]
      \pd{0,0.7};\pd{0.7,0.7};
      \pd{0,0};\pd{0.7,0};
      \draw (0,0) \braidto (0.7,0.7);
      \draw (0.7,0) \braidto (0,0.7);
    \end{tikzpicture}
    \colon 2 \to 2,\ 
    \eta\ =\
    \begin{tikzpicture}[anchorbase]
      \pd{0,0};
      \draw (0,-0.25) to (0,0);
    \end{tikzpicture}
    \colon 0 \to 1,\ 
    \varepsilon\ =\
    \begin{tikzpicture}[anchorbase]
      \pd{0,0};
      \draw (0,0) to (0,0.25);
    \end{tikzpicture}
    \colon 1 \to 0,
  \]
  subject to the following relations and their transforms under $^*$ and $^{\sharp}$:
  \begin{gather}\label{P1}
    \begin{tikzpicture}[anchorbase]
      \pd{0.35,1.4};
      \pd{0,0.7};\pd{0.7,0.7};
      \pd{0,0};
      \draw (0,0) to (0,0.7) \braidto (0.35,1.4);
      \draw (0.7,0.45) to (0.7,0.7) \braidto (0.35,1.4);
    \end{tikzpicture}
    \ =\
    \begin{tikzpicture}[anchorbase]
      \pd{0,0.7};
      \pd{0,0};
      \draw (0,0) to (0,0.7);
    \end{tikzpicture}
    \ , \qquad
    \begin{tikzpicture}[anchorbase]
      \pd{0.7,1.4};\pd{1.4,1.4};
      \pd{0,0.7};\pd{0.7,0.7};\pd{1.4,0.7};
      \pd{0,0};\pd{0.7,0};
      \draw (0,0) to (0,0.7) \braidto (0.7,1.4);
      \draw (0.7,0) to (0.7,0.7) to (0.7,1.4);
      \draw (0.7,0) \braidto (1.4,0.7) to (1.4,1.4);
    \end{tikzpicture}
    \ =\
    \begin{tikzpicture}[anchorbase]
      \pd{0,1.4};\pd{0.7,1.4};
      \pd{0.35,0.7};
      \pd{0,0};\pd{0.7,0};
      \draw (0,0) \braidto (0.35,0.7);
      \draw (0.7,0) \braidto (0.35,0.7);
      \draw (0.35,0.7) \braidto (0,1.4);
      \draw (0.35,0.7) \braidto (0.7,1.4);
    \end{tikzpicture}
    \ ,
    \\ \label{P2}
    \begin{tikzpicture}[anchorbase]
      \pd{0,1.4};\pd{0.7,1.4};
      \pd{0,0.7};\pd{0.7,0.7};
      \pd{0,0};\pd{0.7,0};
      \draw (0,0) \braidto (0.7,0.7) \braidto (0,1.4);
      \draw (0.7,0) \braidto (0,0.7) \braidto (0.7,1.4);
    \end{tikzpicture}
    \ =\
    \begin{tikzpicture}[anchorbase]
      \pd{0,0.7};\pd{0.7,0.7};
      \pd{0,0};\pd{0.7,0};
      \draw (0,0) to (0,0.7);
      \draw (0.7,0) to (0.7,0.7);
    \end{tikzpicture}
    \ , \qquad
    \begin{tikzpicture}[anchorbase]
      \pd{0,2.1};\pd{0.7,2.1};\pd{1.4,2.1};
      \pd{0,1.4};\pd{0.7,1.4};\pd{1.4,1.4};
      \pd{0,0.7};\pd{0.7,0.7};\pd{1.4,0.7};
      \pd{0,0};\pd{0.7,0};\pd{1.4,0};
      \draw (0,0) to (0,0.7) \braidto (0.7,1.4) \braidto (1.4,2.1);
      \draw (0.7,0) \braidto (1.4,0.7) to (1.4,1.4) \braidto (0.7,2.1);
      \draw (1.4,0) \braidto (0.7,0.7) \braidto (0,1.4) to (0,2.1);
    \end{tikzpicture}
    \ =\
    \begin{tikzpicture}[anchorbase]
      \pd{0,2.1};\pd{0.7,2.1};\pd{1.4,2.1};
      \pd{0,1.4};\pd{0.7,1.4};\pd{1.4,1.4};
      \pd{0,0.7};\pd{0.7,0.7};\pd{1.4,0.7};
      \pd{0,0};\pd{0.7,0};\pd{1.4,0};
      \draw (0,0) \braidto (0.7,0.7) \braidto (1.4,1.4) to (1.4,2.1);
      \draw (0.7,0) \braidto (0,0.7) to (0,1.4) \braidto (0.7,2.1);
      \draw (1.4,0) to (1.4,0.7) \braidto (0.7,1.4) \braidto (0,2.1);
    \end{tikzpicture}
    \ ,
    \\ \label{P3}
    \begin{tikzpicture}[anchorbase]
     \pd{0,1.4};\pd{0.7,1.4};
     \pd{0,0.7};\pd{0.7,0.7};
     \pd{0,0};
     \draw (0,0) to (0,0.7) \braidto (0.7,1.4);
     \draw (0.7,0.45) to (0.7,0.7) \braidto (0,1.4);
    \end{tikzpicture}
    \ =\
    \begin{tikzpicture}[anchorbase]
     \pd{0,0.7};\pd{0.7,0.7};
     \pd{0.7,0};
     \draw (0,0.45) to (0,0.7);
     \draw (0.7,0) to (0.7,0.7);
    \end{tikzpicture}
    \ , \qquad
    \begin{tikzpicture}[anchorbase]
     \pd{0,2.1};\pd{0.7,2.1};
     \pd{0,1.4};\pd{0.7,1.4};\pd{1.4,1.4};
     \pd{0,0.7};\pd{0.7,0.7};\pd{1.4,0.7};
     \pd{0,0};\pd{0.7,0};\pd{1.4,0};
     \draw (0,0) to (0,0.7) \braidto (0.7,1.4) to (0.7,2.1);
     \draw (0.7,0) \braidto (1.4,0.7) to (1.4,1.4) \braidto (0.7,2.1);
     \draw (1.4,0) \braidto (0.7,0.7) \braidto (0,1.4) to (0,2.1); 
    \end{tikzpicture}
    \ =\
    \begin{tikzpicture}[anchorbase]
     \pd{0,1.4};\pd{0.7,1.4};
     \pd{0,0.7};\pd{0.7,0.7};
     \pd{-0.7,0};\pd{0,0};\pd{0.7,0};
     \draw (-0.7,0) \braidto (0,0.7);
     \draw (0,0) to (0,0.7) \braidto (0.7,1.4);
     \draw (0.7,0) to (0.7,0.7) \braidto (0,1.4);
    \end{tikzpicture}
    \ ,
    \\ \label{P4}
    \begin{tikzpicture}[anchorbase]
     \pd{0.35,1.4};
     \pd{0,0.7};\pd{0.7,0.7};
     \pd{0,0};\pd{0.7,0};
     \draw (0,0) \braidto (0.7,0.7) \braidto (0.35,1.4);
     \draw (0.7,0) \braidto (0,0.7) \braidto (0.35,1.4);
    \end{tikzpicture}
    \ =\
    \begin{tikzpicture}[anchorbase]
     \pd{0.35,0.7};
     \pd{0,0};\pd{0.7,0};
     \draw (0,0) \braidto (0.35,0.7);
     \draw (0.7,0) \braidto (0.35,0.7);
    \end{tikzpicture}
    \ , \qquad
    \begin{tikzpicture}[anchorbase]
     \pd{0.35,1.4};
     \pd{0,0.7};\pd{0.7,0.7};
     \pd{0.35,0};
     \draw (0.35,0) \braidto (0.7,0.7) \braidto (0.35,1.4);
     \draw (0.35,0) \braidto (0,0.7) \braidto (0.35,1.4);
    \end{tikzpicture}
    \ =\
    \begin{tikzpicture}[anchorbase]
     \pd{0,0.7};
     \pd{0,0};
     \draw (0,0) to (0,0.7);
    \end{tikzpicture}
    \ , \qquad
    \begin{tikzpicture}[anchorbase]
     \pd{0,0};
     \draw (0,-0.25) to (0,0) to (0,0.25);
    \end{tikzpicture}
    \ =\ t1_{0}\ .
  \end{gather}
\end{theo}

It is straightforward to show that every partition diagram can be obtained by compositions and tensor products of the generators $\{\mu,\delta,s,\eta,\epsilon\}$. For example, 
\[
  \begin{tikzpicture}[anchorbase]
    \pd{0.7,1};\pd{1.4,1};\pd{2.1,1};\pd{2.8,1};\pd{3.5,1};  
    \pd{0,0};\pd{0.7,0};\pd{1.4,0};\pd{2.1,0};\pd{2.8,0};\pd{3.5,0};\pd{4.2,0};
    \draw (0,0) \braidto (0.7,1);
    \draw (1.4,0) \braidto (0.7,1);
    \draw (0.7,0) to[out=up,in=up,looseness=1.5] (2.1,0);
    \draw (2.8,0) \braidto (2.1,1);
    \draw (2.8,0) \braidto (3.5,1);
    \draw (4.2,0) ..controls (3.5,0.8) and (2.1,0.2)..  (1.4,1);
  \end{tikzpicture}
  \ =\quad
  \begin{tikzpicture}[anchorbase]
    \pd{0.25,4};\pd{2,4};\pd{2.5,4};\pd{2.75,4};\pd{3,4};
    \pd{0,0};\pd{0.5,0};\pd{1,0};\pd{1.5,0};\pd{2.25,0};\pd{2.62,0};\pd{3,0};
    \draw (0,0) to (0,3) \braidto (0.25,3.5) to (0.25,4);
    \draw (0.5,0) to (0.5,1) \braidto (1,1.5) to[out=up,in=up,looseness=1.5] (1.5,1.5);
    \draw (1,0) to (1,1) \braidto (0.5,1.5) to (0.5,3) \braidto (0.25,3.5);
    \draw (1.5,0) to (1.5,1.5);
    \draw (2.25,0) to (2.25,0.5) \braidto (2,1) to (2,2.5) \braidto (2.5,3) to (2.5,4);
    \draw (2.25,0.5) \braidto (2.5,1) to (2.5,2) \braidto (3,2.5) to (3,4);
    \draw (2.62,0) to (2.62,0.15);
    \draw (3,0) to (3,1) to (3,2) \braidto (2.5,2.5) \braidto (2,3) to (2,4);
    \draw (2.75,3.85) to (2.75,4);
    \draw [dashed] (-0.5, 0.5) -- (3.5,0.5);
    \draw [dashed] (-0.5,1) -- (3.5,1);
    \draw [dashed] (-0.5,1.5) -- (3.5,1.5);
    \draw [dashed] (-0.5,2) -- (3.5,2);
    \draw [dashed] (-0.5,2.5) -- (3.5,2.5);
    \draw [dashed] (-0.5,3) -- (3.5,3);
    \draw [dashed] (-0.5,3.5) -- (3.5,3.5);
  \end{tikzpicture}
  \ =\quad
  \begin{tikzpicture}[anchorbase]
    \pd{0.25,4};\pd{2,4};\pd{2.5,4};\pd{2.75,4};\pd{3,4};
    \pd{0.25,3.5};\pd{2,3.5};\pd{2.5,3.5};\pd{3,3.5};
    \pd{0,3};\pd{0.5,3};\pd{2,3};\pd{2.5,3};\pd{3,3};
    \pd{0,2.5};\pd{0.5,2.5};\pd{2,2.5};\pd{2.5,2.5};\pd{3,2.5};
    \pd{0,2};\pd{0.5,2};\pd{2,2};\pd{2.5,2};\pd{3,2};
    \pd{0,1.5};\pd{0.5,1.5};\pd{1,1.5};\pd{1.5,1.5};\pd{2,1.5};\pd{2.5,1.5};\pd{3,1.5};
    \pd{0,1};\pd{0.5,1};\pd{1,1};\pd{1.5,1};\pd{2,1};\pd{2.5,1};\pd{3,1};
    \pd{0,0.5};\pd{0.5,0.5};\pd{1,0.5};\pd{1.5,0.5};\pd{2.25,0.5};\pd{3,0.5};
    \pd{0,0};\pd{0.5,0};\pd{1,0};\pd{1.5,0};\pd{2.25,0};\pd{2.62,0};\pd{3,0};
    \draw (0,0) to (0,3) \braidto (0.25,3.5) to (0.25,4);
    \draw (0.5,0) to (0.5,1) \braidto (1,1.5) to[out=up,in=up,looseness=1.5] (1.5,1.5);
    \draw (1,0) to (1,1) \braidto (0.5,1.5) to (0.5,3) \braidto (0.25,3.5);
    \draw (1.5,0) to (1.5,1.5);
    \draw (2.25,0) to (2.25,0.5) \braidto (2,1) to (2,2.5) \braidto (2.5,3) to (2.5,4);
    \draw (2.25,0.5) \braidto (2.5,1) to (2.5,2) \braidto (3,2.5) to (3,4);
    \draw (2.62,0) to (2.62,0.15);
    \draw (3,0) to (3,1) to (3,2) \braidto (2.5,2.5) \braidto (2,3) to (2,4);
    \draw (2.75,3.85) to (2.75,4);
  \end{tikzpicture}\ .
\]

In order to prove \cref{partitioncategorypresentation}, the next thing to do is to show that two partition diagrams are equivalent if and only if they can be transformed into each other through relations \cref{P1} to \cref{P4}. The ``if'' part is obvious. However, the ``only if'' part is very technical. We omit this proof here; interested readers should refer to \cite[\S1.4]{Koc04}, \cite[Th.~2.1]{Com16}, and \cite[Prop.~2.1]{Sav19} for more details.

\section{The planar rook category $\PR(t)$ \label{prcategorysection}}

We have defined planar rook diagrams in \cref{introduction}. In this section, we will explore some properties of planar rook diagrams, and we will define the planar rook category $\PR(t)$.

To begin with, once we specify which vertices will be isolated vertices (or, equivalently, specify which vertices are in some size $2$ blocks) in a planar rook diagram, that planar rook diagram is completely determined. (See \cite[\S1]{Her06} for a precise proof of this argument.) For example, consider a planar rook partition $D\colon 5\to7$. Once we specify $\{3,5,2',3',4',6'\}$ will be the isolated vertices, this planar rook partition can only be $\big\{ \{1,1'\} , \{2,5'\} , \{4,7'\} , \{3\} , \{5\} ,\\ \{2'\} , \{3'\} , \{4'\} , \{6\} \big\}$, and it is depicted as follows:
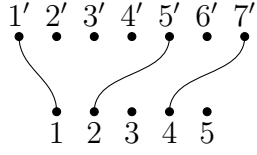
\[
  \begin{tikzpicture}[anchorbase]
    \pd{0,1} node[anchor=south] {$1'$};
    \pd{0.5,1} node[anchor=south] {$2'$};
    \pd{1,1} node[anchor=south] {$3'$};
    \pd{1.5,1} node[anchor=south] {$4'$};
    \pd{2,1} node[anchor=south] {$5'$};
    \pd{2.5,1} node[anchor=south] {$6'$};  
    \pd{3,1} node[anchor=south] {$7'$};
    \pd{0.5,0} node[anchor=north] {$1$};
    \pd{1,0} node[anchor=north] {$2$};
    \pd{1.5,0} node[anchor=north] {$3$};
    \pd{2,0} node[anchor=north] {$4$};
    \pd{2.5,0} node[anchor=north] {$5$};
    \draw (0.5,0) \braidto (0,1);
    \draw (1,0) \braidto (2,1);
    \draw (2,0) \braidto (3,1);
  \end{tikzpicture}
\]

\begin{prop}\label{priffprop}
  A rook diagram is planar if and only if it is a tensor product of rook diagrams with a single block.
\end{prop}
\begin{proof}
  There are only three rook diagrams with a single block, which are shown below.
  
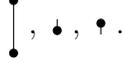
\[
    \begin{tikzpicture}[anchorbase]
      \pd{0,0.7};
      \pd{0,0};
      \draw (0,0) to (0,0.7);
    \end{tikzpicture}
    \ ,\ 
    \begin{tikzpicture}[anchorbase]
      \pd{0,0};
      \draw (0,0) to (0,0.15);
    \end{tikzpicture}
    \ ,\ 
    \begin{tikzpicture}[anchorbase]
      \pd{0,0};
      \draw (0,-0.15) to (0,0);
    \end{tikzpicture}\ .
  \]

  The ``if'' part is obvious. 

  Now let us prove the ``only if'' part. Let $D\colon k\to l$ be a planar rook diagram with $r$ size two blocks. Suppose the non-isolated vertices in the bottom row of $D$ are labeled by $p_1<\cdots<p_r$, and the non-isolated vertices in the top row of $D$ are labeled by $q'_1<\cdots<q'_r$. Then $D$ corresponds to the following planar rook partition:
  \begin{multline*}
    \big\{ \{1\},\ \cdots,\ \{p_1-1\},\ \{1'\},\ \cdots,\ \{(q_1-1)'\},\ \{p_1,q'_1\},\ \\\{p_1+1\},\ \cdots,\ \{p_2-1\},\ \{(q_1+1)'\},\ \cdots,\ \{(q_2-1)'\},\ \cdots\cdots \big\}, 
  \end{multline*}
  and $D$ is just the tensor product of the above blocks from left to right one by one.
\end{proof}

Fix $t\in\kk$. We will define the planar rook category $\PR(t)$ in terms of generators and relations, and we will show that morphism spaces of $\PR(t)$ have linear bases given by all planar rook diagrams. This is why we call $\PR(t)$ the planar rook category.

\begin{defin}\label{prcategory}
  Fix $t\in\kk$. We define the \emph{planar rook category} $\PR(t)$ to be the strict $\kk$-linear monoidal category generated by the object $1$ and the morphisms
  \begin{equation}\label{prgenerator}
    \eta\ =\
    \begin{tikzpicture}[anchorbase]
      \pd{0,0};
      \draw (0,-0.25) to (0,0);
    \end{tikzpicture}
    \colon 0 \to 1, \quad
    \varepsilon\ =\
    \begin{tikzpicture}[anchorbase]
      \pd{0,0};
      \draw (0,0) to (0,0.25);
    \end{tikzpicture}
    \colon 1 \to 0,
  \end{equation}
  subject to the following relation:
  \begin{equation}\label{PR1}
    \begin{tikzpicture}[anchorbase]
      \pd{0,0};
      \draw (0,-0.25) to (0,0) to (0,0.25);
    \end{tikzpicture}
    \ =\ t1_{0}.
  \end{equation}
\end{defin}

\begin{cor}\label{prin}
  Every planar rook diagram can be built by a tensor product of the generators $\{\eta,\varepsilon\}$ given in \cref{prgenerator}.
\end{cor}
\begin{proof}
  This is just the ``only if'' part of \cref{priffprop}.
\end{proof}

For example, consider the following planar rook diagram.
\[
  \begin{tikzpicture}[anchorbase]
    \pd{0,1};\pd{0.5,1};\pd{1,1};\pd{1.5,1};\pd{2,1};\pd{2.5,1};\pd{3,1};
    \pd{0.5,0};\pd{1,0};\pd{1.5,0};\pd{2,0};\pd{2.5,0};
    \draw (0.5,0) \braidto (0,1);
    \draw (1,0) \braidto (2,1);
    \draw (2,0) \braidto (3,1);
    \draw (1.5,0) to (1.5,0.15);
    \draw (2.5,0) to (2.5,0.15);
    \draw (0.5,0.8) to (0.5,1);
    \draw (1,0.85) to (1,1);
    \draw (1.5,0.85) to (1.5,1);
    \draw (2.5,0.85) to (2.5,1);
  \end{tikzpicture}
  \ = \ 
  \begin{tikzpicture}[anchorbase]
    \pd{0,0.7};
    \pd{0,0};
    \draw (0,0) to (0,0.7);
  \end{tikzpicture}
  \ \otimes\ 
  \begin{tikzpicture}[anchorbase]
    \pd{0,0};
    \draw (0,-0.25) to (0,0);
  \end{tikzpicture}
  \ \otimes\ 
  \begin{tikzpicture}[anchorbase]
    \pd{0,0};
    \draw (0,-0.25) to (0,0);
  \end{tikzpicture}
  \ \otimes\ 
  \begin{tikzpicture}[anchorbase]
    \pd{0,0};
    \draw (0,-0.25) to (0,0);
  \end{tikzpicture}
  \ \otimes\ 
  \begin{tikzpicture}[anchorbase]
    \pd{0,0.7};
    \pd{0,0};
    \draw (0,0) to (0,0.7);
  \end{tikzpicture}
  \ \otimes\ 
  \begin{tikzpicture}[anchorbase]
    \pd{0,0};
    \draw (0,-0.25) to (0,0);
  \end{tikzpicture}
  \ \otimes\ 
  \begin{tikzpicture}[anchorbase]
    \pd{0,0};
    \draw (0,0) to (0,0.25);
  \end{tikzpicture}
  \ \otimes\ 
  \begin{tikzpicture}[anchorbase]
    \pd{0,0.7};
    \pd{0,0};
    \draw (0,0) to (0,0.7);
  \end{tikzpicture}
  \ \otimes\ 
  \begin{tikzpicture}[anchorbase]
    \pd{0,0};
    \draw (0,0) to (0,0.25);
  \end{tikzpicture}
\]

Some readers may worry about the order of the appearance of $\eta$ and $\varepsilon$ in the above construction. However, we have the following proposition, which shows that $\eta$ and $\varepsilon$ commute so that there is no need to worry about the order of $\eta$ and $\varepsilon$.

\begin{prop}\label{prmorerelations}
  The following relation holds in the planar rook category $\PR(t)$:
  \begin{equation}\label{PR2}
    \begin{tikzpicture}[{baseline={(0,0.25)}}]
      \pd{0,0.7};
      \draw (0,0.45) to (0,0.7);
    \end{tikzpicture}
    \otimes
    \begin{tikzpicture}[{baseline={(0,0.25)}}]
      \pd{0,0};
      \draw (0,0) to (0,0.25);
    \end{tikzpicture}
    \ =\
    \begin{tikzpicture}[{baseline={(0,0.25)}}]
      \pd{0,0};
      \draw (0,0) to (0,0.25);
    \end{tikzpicture}
    \otimes
    \begin{tikzpicture}[{baseline={(0,0.25)}}]
      \pd{0,0.7};
      \draw (0,0.45) to (0,0.7);
    \end{tikzpicture}
  \end{equation}
\end{prop}
\begin{proof}
  This is just a corollary of the interchange law given in \cref{interchangelaw}.
\end{proof}

\begin{lem}\label{prcompose}
  Let $D\colon l\to m$ and $D'\colon k\to l$ be two planar rook diagrams. Then the diagram $D'':=D\circ D'$ is equal to a power of $t$ times a planar rook diagram.
\end{lem}
\begin{proof}
  We make the following local substitutions in $D''$.
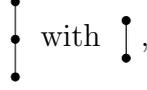
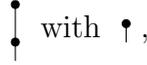
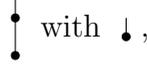
  \begin{enumerate}
    \renewcommand{\theenumi}{(\alph{enumi})}
    \item\label{prreplacea} replace every
    \[\ \ \begin{tikzpicture}[anchorbase]
      \pd{0,1};
      \pd{0,0.5};
      \pd{0,0};
      \draw (0,0) to (0,0.5) to (0,1);
    \end{tikzpicture}\ \ 
    \text{with}        
    \ \ \begin{tikzpicture}[anchorbase]
      \pd{0,0.5};
      \pd{0,0};
      \draw (0,0) to (0,0.5);
    \end{tikzpicture}\ ,\]
  \item\label{prreplaceb} replace every
  \[\ \ \begin{tikzpicture}[anchorbase]
    \pd{0,0.5};
    \pd{0,0};
    \draw (0,-0.25) to (0,0) to (0,0.5);
  \end{tikzpicture}\ \ 
  \text{with}
  \ \ \begin{tikzpicture}[anchorbase]
    \pd{0,0};
    \draw (0,-0.25) to (0,0);
  \end{tikzpicture}\ ,\]
  \item\label{prreplacec} replace every
  \[\ \ \begin{tikzpicture}[anchorbase]
    \pd{0,0.5};
    \pd{0,0};
    \draw (0,0) to (0,0.5) to (0,0.75);
  \end{tikzpicture}\ \ 
  \text{with} 
  \ \ \begin{tikzpicture}[anchorbase]
    \pd{0,0};
    \draw (0,0) to (0,0.25);
  \end{tikzpicture}\ ,\]
  \item\label{prreplaced} replace every
  $\ \begin{tikzpicture}[anchorbase]
    \pd{0,0};
    \draw (0,-0.15) to (0,0) to (0,0.15);
  \end{tikzpicture}\ $
  with a scalar multiple $t$.
\end{enumerate}
After steps \cref{prreplacea} to \cref{prreplaced}, what we obtain is a power of $t$ times a planar rook diagram.
\end{proof}
For example, if
\[
  D\ =\ 
  \begin{tikzpicture}[anchorbase]
    \pd{0.25,1};\pd{0.75,1};\pd{1.25,1};\pd{1.75,1};\pd{2.25,1};
    \pd{0,0.5};\pd{0.5,0.5};\pd{1,0.5};\pd{1.5,0.5};\pd{2,0.5};\pd{2.5,0.5};
    \draw (1,0.5) \braidto (1.25,1);
    \draw (1.5,0.5) to (1.5,0.65);
    \draw (0,0.5) to (0,0.65);
    \draw (0.5,0.5) \braidto (0.75,1);
    \draw (2,0.5) to (2,0.65);
    \draw (2.5,0.5) \braidto (1.75,1);
    \draw (0.25,0.85) to (0.25,1);
    \draw (2.25,0.85) to (2.25,1);
  \end{tikzpicture}
  \ ,\ \text{and}\quad
  D'\ =\ 
  \begin{tikzpicture}[anchorbase]
    \pd{0,0.5};\pd{0.5,0.5};\pd{1,0.5};\pd{1.5,0.5};\pd{2,0.5};\pd{2.5,0.5};
    \pd{0.5,0};\pd{1,0};\pd{1.5,0};\pd{2,0};
    \draw (0.5,0) \braidto (1,0.5);
    \draw (1,0) \braidto (1.5,0.5);
    \draw (1.5,0) to (1.5,0.15);
    \draw (2,0) to (2,0.15);
    \draw (0,0.35) to (0,0.5);
    \draw (0.5,0.35) to (0.5,0.5);
    \draw (2,0.35) to (2,0.5);
    \draw (2.5,0.35) to (2.5,0.5);
  \end{tikzpicture}
  \ ,
\]
then we have
\[
  D'':=D\circ D'\ =\ 
  \begin{tikzpicture}[anchorbase]
    \pd{0.25,1};\pd{0.75,1};\pd{1.25,1};\pd{1.75,1};\pd{2.25,1};
    \pd{-0.125,0.5};\pd{0.375,0.5};\pd{0.875,0.5};\pd{1.5,0.5};\pd{2,0.5};\pd{2.5,0.5};
    \pd{0.5,0};\pd{1,0};\pd{1.5,0};\pd{2,0};
    \draw (0.5,0) \braidto (1.25,1);
    \draw (1,0) \braidto (1.5,0.5) to (1.5,0.65);
    \draw (1.5,0) to (1.5,0.15);
    \draw (2,0) to (2,0.15);
    \draw (-0.125,0.35) to (-0.125,0.5) to (-0.125,0.65);
    \draw (0.375,0.35) to (0.375,0.5) \braidto (0.75,1);
    \draw (2,0.35) to (2,0.5) \braidto (2,0.65);
    \draw (2.5,0.35) to (2.5,0.5) \braidto (1.75,1);
    \draw (0.25,0.85) to (0.25,1);
    \draw (2.25,0.85) to (2.25,1);
  \end{tikzpicture}
  \ =\ t^2\ 
  \begin{tikzpicture}[anchorbase]
    \pd{0.25,0.7};\pd{0.75,0.7};\pd{1.25,0.7};\pd{1.75,0.7};\pd{2.25,0.7};
    \pd{0.5,0};\pd{1,0};\pd{1.5,0};\pd{2,0};
    \draw (0.5,0) \braidto (1.25,0.7);
    \draw (1,0) to (1,0.15);
    \draw (1.5,0) to (1.5,0.15);
    \draw (2,0) to (2,0.15);
    \draw (0.25,0.55) to (0.25,0.7);
    \draw (0.75,0.55) to (0.75,0.7);
    \draw (1.75,0.55) to (1.75,0.7);
    \draw (2.25,0.55) to (2.25,0.7);
  \end{tikzpicture}
  \ .
\]

\begin{prop}\label{prcomb}
  Every morphism in $\Hom_{\PR(t)}(k,l)$ is a $\kk$-linear combination of planar rook diagrams of type $\binom{l}{k}$.
\end{prop}
\begin{proof}
  By definition, every morphism $f\in\Hom_{\PR(t)}(k,l)$ is of the form $f=\sum_{i=1}^{n}a_iD_i$, where $n$ is some nonnegative integer, $a_i\in\kk$, and $D_i\colon k_i\to l_i$ is a diagram obtained by compositions and tensor products of the generators $\{\eta,\varepsilon\}$ for all $i$. Clearly, each $D_i$ has a decomposition
  \[
    D_i=D_i^{1}\circ\cdots\circ D_i^{m_i},
  \]
  where $m_i$ is some nonnegative integer, and each $D_i^j$ is a tensor product of the generators $\{\eta,\varepsilon\}$ for all $i,j$; thus a planar rook diagram for all $i,j$. Then by using \cref{prcompose} $m_i-1$ times, we get that
  \[
    D_i=t^{\alpha_i}D'_i,
  \]
  where $\alpha_i$ is a nonnegative integer, and $D'_i$ is a planar rook diagram for all $i$. Therefore.
  \[
    f=\sum_{i=1}^{n}a_iD_i=\sum_{i=1}^{n}a_it^{\alpha_i}D'_i
  \]
  is a linear combination of planar rook diagrams.
\end{proof}

\begin{prop}\label{prindependent}
  All the planar rook diagrams of type $\binom{l}{k}$ are linearly independent in the morphism space $\Hom_{\PR(t)}(k,l)$.
\end{prop}
\begin{proof}
  We can define a strict $\kk$-linear monoidal functor $F\colon\PR(t)\to\Par(t)$ by 
  \[
    1\mapsto 1,\quad \eta\mapsto\eta,\quad \varepsilon\mapsto\varepsilon.
  \]
  The functor $F$ is well-defined because the relation \cref{PR1} holds in $\Par(t)$ (the relation \cref{P4}). In addition, the functor $F$ induces a $\kk$-module homomorphism between the morphism spaces
  \[
    F_{k,l}\colon \Hom_{\PR(t)}(k,l)\to\Hom_{\Par(t)}(k,l),
  \]
  which maps each planar rook diagram of type $\binom{l}{k}$ in $\PR(t)$ to the same planar rook diagram in $\Par(t)$. It is a well-known fact in linear algebra that if the images of some elements under a linear map are linearly independent, then the elements are linearly independent in the domain. Therefore, all planar rook diagrams of type $\binom{l}{k}$ are linearly independent in $\Hom_{\PR(t)}(k,l)$. 
\end{proof}

\begin{theo}\label{prbasistheorem}
  The morphism space $\Hom_{\PR(t)}(k,l)$ has a linear basis given by all planar rook diagrams of type $\binom{l}{k}$.
\end{theo}
\begin{proof}
  The proof is a combination of \cref{prin}, \cref{prcomb}, and \cref{prindependent}.
\end{proof}

\begin{cor}
  Fix $k\in\N$. The endomorphism algebra $\End_{\PR(t)}(k)$ is the \emph{planar rook algebra} $PR_k(t)$ given in \cref{introduction}.
\end{cor}
\begin{proof} This is a corollary of \cref{prbasistheorem}.
\end{proof}

\section{The rook category $\cR(t)$ \label{rookcategorysection}}

We have defined rook diagrams in \cref{introduction}. In this section, we will discuss some properties of rook diagrams, and we will define the rook category $\cR(t)$.

\subsection{Rook matrices\label{rookoriginsubsection}}
Let $M_{l\times k}(\kk)$ denote the set of matrices with $l$ rows and $k$ columns over a commutative ring $\kk$. A matrix $A\in M_{l\times k}(\kk)$ is called a \emph{rook matrix} of size $l\times k$ if each entry of $A$ is equal to either $0$ or $1$, and $A$ has at most one entry equal to $1$ in each row and column. The reason why they are called the rook matrices is that such matrices look like non-attacking rooks on a $l\times k$ chessboard. For example, 
\[
  \begin{bmatrix}
    0 & 1 & 0 & 0 & 0 \\
    0 & 0 & 0 & 0 & 1 \\
    0 & 0 & 0 & 0 & 0 \\
    0 & 0 & 1 & 0 & 0 \\
  \end{bmatrix}
\]
is a rook matrix of size $4\times 5$. Let $R_{l\times k}$ denote the collection of all rook matrices in $M_{l\times k}(\kk)$. A \emph{permutation matrix} of size $k\times k$ is a matrix obtained by permuting the rows of the $k\times k$ identity matrix. Let $S_{k\times k}$ denote the collection of all permutation matrices in $M_{k\times k}(\kk)$. Clearly, we have $S_{k\times k}\subseteq R_{k\times k}$, and $S_{k\times k}$ are the only full-rank ones in $R_{k\times k}$. 

Actually, there are very nice correspondences between rook diagrams and rook matrices. Let $D\colon k\to l$ be a rook diagram with $r$ size two blocks. Suppose the non-isolated vertices in the bottom row of $D$ are labeled by $p_1<\cdots<p_r$, and the non-isolated vertices in the top row of $D$ are labeled by $q'_1<\cdots<q'_r$. We can associate $D$ with a rook matrix $[D]\in R_{l\times k}$ by the following map:
\begin{equation}\label{matrixmap}
  D \mapsto [D]=\sum_{i=1}^{r}E_{q_i,p_i}^{l\times k},
\end{equation}
where $E^{l\times k}_{m,n}$ is a matrix in $M_{l\times k}(\kk)$ with an entry $1$ in the $(m,n)$ position and $0's$ elsewhere. For example, the rook diagram
\[
  D\ =\ 
  \begin{tikzpicture}[anchorbase]
    \pd{0.25,1};\pd{0.75,1};\pd{1.25,1};\pd{1.75,1};
    \pd{0,0};\pd{0.5,0};\pd{1,0};\pd{1.5,0};\pd{2,0};
    \draw (0.5,0) \braidto (0.25,1);
    \draw (1,0) \braidto (1.75,1);
    \draw (2,0)\braidto(0.75,1);
  \end{tikzpicture}\ ,
\]
is mapped to the rook matrix 
\[
  [D]\ =\ 
  \begin{bmatrix}
    0 & 1 & 0 & 0 & 0 \\
    0 & 0 & 0 & 0 & 1 \\
    0 & 0 & 0 & 0 & 0 \\
    0 & 0 & 1 & 0 & 0 \\
  \end{bmatrix}\ .
\]
Moreover, it is easy to check that the map \cref{matrixmap} is bijective, and the number of size $2$ blocks in $D$ is equal to the rank of $[D]$. 

A rook diagram of type $\binom{k}{k}$ is called a \emph{permutation diagram} if all of its blocks have size $2$. Let $S_k(t)$ denote the collection of all permutation diagrams of type $\binom{k}{k}$.

If we restrict the map \cref{matrixmap} to $S_k(t)$, we get an algebra isomorphism between $S_k(t)$ and $S_{k\times k}$ (no matter what $t$ is). If $l=k$ and $t=1$, the map \cref{matrixmap} becomes an algebra isomorphism between the rook algebra $R_k(1)$ and $R_{k\times k}$. 

For an arbitrary matrix, the first nonzero entry in each row is called a \emph{leading entry}. \color{black}Inspired by the echelon form of a matrix, we define the pseudo-echelon form of a rook matrix.
\begin{defin}\label{pseudoechelonform}
  A rook matrix $[D]$ is said to be in \emph{pseudo-echelon form} if each leading entry of $[D]$ is in a column to the right of all leading entries above it.
\end{defin}

Contrary to the traditional echelon forms, we do not distinguish row and column pseudo-echelon forms because they are the same due to the following two facts:
\begin{itemize}
  \item There is only one nonzero entry in each nonzero row and column of a rook matrix. 
  \item We do not require the zero rows to be below the nonzero rows in a rook matrix in pseudo-echelon form.
\end{itemize}

Apart from \cref{priffprop}, we have another characterization of planar rook diagrams.
\begin{prop}\label{pseudoiffprop}
  A rook diagram $D$ is planar if and only if the rook matrix $[D]$ is in pseudo-echelon form.
\end{prop}
\begin{proof}
  Let $D\colon k\to l$ be a rook diagram with $r$ size two blocks. Suppose the non-isolated vertices in the bottom row of $D$ are labeled by $p_1<\cdots<p_r$, and the non-isolated vertices in the top row of $D$ are labeled by $q'_1<\cdots<q'_r$. Then all the nonzero entries of $[D]$ are in the positions $\{(q_i,p_i)\ |\ 1\leq i\leq r\}$. The rook diagram $D$ is planar if and only if
  \begin{equation}\label{planarechelon1}
    p_i<p_j\Leftrightarrow q'_i<q'_j;
  \end{equation}
  the rook matrix $[D]$ is in pseudo-echelon form if and only if
  \begin{equation}\label{planarechelon2}
    q_i<q_j\Leftrightarrow p_i<p_j,
  \end{equation}
  The two conditions \cref{planarechelon1} and \cref{planarechelon2} are the same. Therefore, the rook diagram $D$ is planar if and only if the rook matrix $[D]$ is in pseudo-echelon form.
\end{proof}

\begin{prop}\label{rookmatrixsp}
  Every $[D]\in R_{l\times k}$ can be changed to a pseudo-echelon form by permuting the nonzero rows, which is also saying that $[D]$ has a decomposition $[D]=[S][P]$, where $[S]$ is a permutation matrix, and $[P]$ is a rook matrix in pseudo-echelon form. Moreover, the rook matrix $[P]$ has the same zero rows and columns as $[D]$.
\end{prop}
\begin{proof}
  Suppose all the nonzero entries in $[D]$ are in the positions $\{(q_i,p_{h_i})\ |\  1\leq i\leq r\}$. By changing the row $q_i$ to the row $q_{h_i}$ for all $i$, we get a rook matrix $[P]$ whose nonzero entries are in the positions $\{(q_i,p_i)\ |\ 1\leq i\leq r\}$. Moreover, we have not changed the zero rows and zero columns of $[D]$. Therefore, the rook matrix $[P]$ is in pseudo-echelon form, and it has the same zero rows and columns as $[D]$. By noticing that permuting the rows corresponds to multiplying on the left by a permutation matrix, we can finish the proof of this lemma.
\end{proof}
For example, the rook matrix
\[
  [D]\ =\ 
  \begin{bmatrix}
    0 & 1 & 0 & 0 & 0 \\
    0 & 0 & 0 & 0 & 1 \\
    0 & 0 & 0 & 0 & 0 \\
    0 & 0 & 1 & 0 & 0 \\
  \end{bmatrix}
\]
can be changed to a rook matrix in pseudo-echelon form
\[
  [P]\ =\ 
  \begin{bmatrix}
    0 & 1 & 0 & 0 & 0 \\
    0 & 0 & 1 & 0 & 0 \\
    0 & 0 & 0 & 0 & 0 \\
    0 & 0 & 0 & 0 & 1 \\
  \end{bmatrix}
\]
by permuting the $2$-th row ($q_2$-th row) and the $4$-th row (the $q_3$-th row).

Let $A^t$ denote the transpose of a matrix $A$. It is easy to see that the transposes of a rook matrix, a permutation matrix, and a rook matrix in pseudo-echelon form are still a rook matrix, a permutation matrix, and a rook-matrix in pseudo-echelon form respectively. Moreover, the zero rows in $A$ are in one-to-one correspondences to the zero columns in $A^t$. Therefore, we have the following corollary.
\begin{cor}\label{rookmatrixps}
  Every rook matrix $[D]$ has a decomposition $[D]=[P][S]$, where $[S]$ is a permutation matrix, and $[P]$ is a matrix in pseudo-echelon form that has the same zero rows and columns as $[D]$.
\end{cor}

\begin{proof}
  By \cref{rookmatrixsp}, we know that there exists a permutation matrix $[S]$ and a rook matrix $[P]$ in pseudo-echelon form that has the same zero rows and columns as $[D]^t$ such that
  \[
    [D]^t=[S][P].
  \]
  Take transpose of both sides of the above equality, we get
  \[
    [D]=([D]^t)^t=[P]^t[S]^t,
  \]
  which finishes the proof since $[S]^t$ is a permutation matrix and $[P]^t$ is a rook matrix in pseudo-echelon form that has the same zero rows and columns as $[D]$.
\end{proof}

Note that the rook matrices in pseudo-echelon forms appeared in \cref{rookmatrixsp} and \cref{rookmatrixps} are the same because a rook matrix in pseudo-echelon form is completely determined by its zero rows and columns. (See the similar argument for planar rook diagrams given in the beginning of \cref{prcategorysection}.) Furthermore, there are very nice correspondences between rook diagrams and rook matrices given in the following table. 
\begin{center}
\begin{tabular}{c|c}
  \hline
  \hline
   rook diagrams & rook matrices \\
  \hline
   permutation diagrams & permutations matrices \\
  \hline
   planar rook diagrams & rook matrices in pseudo-echelon form \\
  \hline
   flip in the horizontal axis (involution $^*$) & transpose \\
  \hline
  \hline
\end{tabular}
\end{center}
Therefore, as a corollary of \cref{rookmatrixsp}, \cref{rookmatrixps}, and the above table, we have that every rook diagram $D$ has decompositions $D=S\circ P$ and $D=P\circ S'$ where $S,S'$ are permutation diagrams, and $P$ is a planar rook diagram that is of the same type of $D$, and that has the same isolated vertices as $D$. This also provides an alternative proof for \cref{decompose2}\cref{decompose2a} and \cref{decompose2cor}\cref{decompose2cora}.

\subsection{The rook category $\cR(t)$}
\begin{defin}\label{rookcategory} 
  Fix $t\in\kk$. We define the \emph{rook category} $\cR(t)$ to be the strict $\kk$-linear monoidal category generated by the object $1$ and morphisms
  \begin{equation}\label{rookgenerator}
     s =
   \begin{tikzpicture}[anchorbase]
     \pd{0,0};
     \pd{0.7,0};
     \pd{0,0.7};
     \pd{0.7,0.7};
     \draw (0,0) \braidto (0.7,0.7);
     \draw (0.7,0) \braidto (0,0.7);
   \end{tikzpicture}
   \colon 2 \to 2,\quad
   \eta =
   \begin{tikzpicture}[anchorbase]
     \pd{0,0.5};
     \draw(0,0.25) to (0,0.5);
   \end{tikzpicture}
   \ \colon 0 \to 1,\quad
   \varepsilon =
   \begin{tikzpicture}[anchorbase]
     \pd{0,0};
     \draw (0,0.25) to (0,0);
   \end{tikzpicture}
   \ \colon 1 \to 0,
  \end{equation}
  subject to the following relations and their transforms under $^*$ and $^{\sharp}$:
  \begin{equation} \label{R1}
   \begin{tikzpicture}[anchorbase]
     \pd{0,0};
     \pd{0.7,0};
     \pd{0,0.7};
     \pd{0.7,0.7};
     \pd{0,1.4};
     \pd{0.7,1.4};
     \draw (0,0) \braidto (0.7,0.7) \braidto (0,1.4);
     \draw (0.7,0) \braidto (0,0.7) \braidto (0.7,1.4);
   \end{tikzpicture}
   \ =\
   \begin{tikzpicture}[anchorbase]
     \pd{0,0};
     \pd{0.7,0};
     \pd{0,0.7};
     \pd{0.7,0.7};
     \draw (0,0) to (0,0.7);
     \draw (0.7,0) to (0.7,0.7);
   \end{tikzpicture}
   \ ,\qquad
   \begin{tikzpicture}[anchorbase]
     \pd{0,0};
     \pd{0.7,0};
     \pd{1.4,0};
     \pd{0,0.7};
     \pd{0.7,0.7};
     \pd{1.4,0.7};
     \pd{0,1.4};
     \pd{0.7,1.4};
     \pd{1.4,1.4};
     \pd{0,2.1};
     \pd{0.7,2.1};
     \pd{1.4,2.1};
     \draw (0,0) to (0,0.7) \braidto (0.7,1.4) \braidto (1.4,2.1);
     \draw (0.7,0) \braidto (1.4,0.7) to (1.4,1.4) \braidto (0.7,2.1);
     \draw (1.4,0) \braidto (0.7,0.7) \braidto (0,1.4) to (0,2.1);
   \end{tikzpicture}
   \ =\
   \begin{tikzpicture}[anchorbase]
     \pd{0,0};
     \pd{0.7,0};
     \pd{1.4,0};
     \pd{0,0.7};
     \pd{0.7,0.7};
     \pd{1.4,0.7};
     \pd{0,1.4};
     \pd{0.7,1.4};
     \pd{1.4,1.4};
     \pd{0,2.1};
     \pd{0.7,2.1};
     \pd{1.4,2.1};
     \draw (0,0) \braidto (0.7,0.7) \braidto (1.4,1.4) to (1.4,2.1);
     \draw (0.7,0) \braidto (0,0.7) to (0,1.4) \braidto (0.7,2.1);
     \draw (1.4,0) to (1.4,0.7) \braidto (0.7,1.4) \braidto (0,2.1);
   \end{tikzpicture}
   \ ,
 \end{equation}
 \begin{equation}\label{R2}
   \begin{tikzpicture}[anchorbase]
     \pd{0,0.7};
     \pd{0,1.4};
     \pd{0.7,0};
     \pd{0.7,0.7};
     \pd{0.7,1.4};
     \draw (0,0.45) to (0,0.7) \braidto (0.7,1.4);
     \draw (0.7,0) to (0.7,0.7) \braidto (0,1.4);
   \end{tikzpicture}
   \ =\
   \begin{tikzpicture}[anchorbase]
     \pd{0,0};
     \pd{0,0.7};
     \pd{0.7,0.7};
     \draw (0,0) to (0,0.7);
     \draw (0.7,0.45) to (0.7,0.7);
   \end{tikzpicture}
   \ ,
 \end{equation}
 \begin{equation}\label{R3}
   \begin{tikzpicture}[anchorbase]
     \pd{0,0};
     \draw (0,-0.2) to (0,0.2);
   \end{tikzpicture}
   = t 1_0.
 \end{equation}
\end{defin}

\begin{prop}\label{rookmorerealtions}
  The following relation and its transform under $^*$ holds in the rook category $\cR(t)$.
  \begin{equation}\label{R4}
    \begin{tikzpicture}[anchorbase]
      \pd{0,1.4};\pd{0.7,1.4};
      \pd{0,0.7};\pd{0.7,0.7};
      \draw (0,0.45) to (0,0.7) \braidto (0.7,1.4);
      \draw (0.7,0.45) to (0.7,0.7) \braidto (0,1.4);
    \end{tikzpicture}
    \ =\
    \begin{tikzpicture}[anchorbase]
      \pd{0,0.7};\pd{0.7,0.7};
      \draw (0,0.45) to (0,0.7);
      \draw (0.7,0.45) to (0.7,0.7);
    \end{tikzpicture}
    \ .\qedhere
  \end{equation}
\end{prop}

\begin{proof}
For \cref{R4},
  \[
    \begin{tikzpicture}[anchorbase]
      \pd{0,1.4};\pd{0.7,1.4};
      \pd{0,0.7};\pd{0.7,0.7};
      \draw (0,0.45) to (0,0.7) \braidto (0.7,1.4);
      \draw (0.7,0.45) to (0.7,0.7) \braidto (0,1.4);
    \end{tikzpicture}
    \ \overset{\cref{ax4}}{=\joinrel=}\
    \begin{tikzpicture}[anchorbase]
      \pd{0,1.4};\pd{0.7,1.4};
      \pd{0,0.7};\pd{0.7,0.7};
      \pd{0.7,0};
      \draw (0,0.45) to (0,0.7) \braidto (0.7,1.4);
      \draw (0.7,-0.25) to (0.7,0) to (0.7,0.7) \braidto (0,1.4);
    \end{tikzpicture}
    \ \overset{\cref{R2}}{=\joinrel=}\ 
    \begin{tikzpicture}[anchorbase]
      \pd{0,0.7};\pd{0.7,0.7};
      \pd{0,0};
      \draw (0,-0.15) to (0,0) to (0,0.7);
      \draw (0.7,0.45) to (0.7,0.7);
    \end{tikzpicture}
    \ \overset{\cref{ax4}}{=\joinrel=}\
    \begin{tikzpicture}[anchorbase]
      \pd{0,0.7};\pd{0.7,0.7};
      \draw (0,0.45) to (0,0.7);
      \draw (0.7,0.45) to (0.7,0.7);
    \end{tikzpicture}
    \ .\qedhere
  \]
\end{proof}

We will prove that morphism spaces of the rook category $\cR(t)$ have linear bases given by all rook diagrams.

\begin{prop}\label{rookin}
  Every rook diagram can be built by compositions and tensor products of the generators $\{s,\eta,\varepsilon\}$ given in \cref{rookgenerator}.
\end{prop}
\begin{proof}
  By \cref{symmetricgroupcategory}, every permutation diagram can be built by compositions and tensor products of the generator $s$. By \cref{prin}, every planar rook diagram can be built by a tensor product of the generators $\{\varepsilon,\eta\}$. By \cref{decompose2}\cref{decompose2a}, every rook diagram can be decomposed into a composition of a permutation diagram and a planar rook diagram; therefore, every rook diagram can be built by the generators $\{s,\eta,\varepsilon\}$.
\end{proof}

\begin{lem}\label{rookchange}
  Consider a diagram $S\circ P$, where $S$ is a permutation diagram, and $P$ is a planar rook diagram. Then there exists a permutation diagram $S'$ and a planar rook diagram $P'$ such that $S\circ P=P'\circ S'$.
\end{lem}
\begin{proof}
  The diagram $S\circ P$ is a rook diagram due to relation \cref{R2}. Then by \cref{decompose2cor}\cref{decompose2cora}, there exists a permutation diagram $S'$ and a planar rook diagram $P'$ such that $S\circ P=P'\circ S'$. 
\end{proof}

\begin{lem}\label{rookcompose}
  Let $D\colon l\to m$ and $D'\colon k\to l$ be two rook diagrams. Then the diagram $D'':= D\circ D'$ is equal to a power of $t$ times a rook diagram.
\end{lem}
\begin{proof}
  By \cref{decompose2}\cref{decompose2a} and \cref{decompose2cor}\cref{decompose2cora}, the rook diagrams $D$ and $D'$ have decompositions
  \[
    D=S\circ P,\quad D'=P'\circ S',
  \]
  where $S,S'$ are permutation diagrams, and $P,P'$ are planar rook diagrams. Then we have
  \begin{align*}
  D'' &= D\circ D' \\
      &= S\circ P\circ P'\circ S' \\
      &= S\circ (t^{\alpha}\ P'')\circ S' & \text{Lemma~\ref{prcompose}}\\
      &= t^{\alpha}\ (S\circ P''\circ S') \\
      &= t^{\alpha}\ (S\circ S''\circ P'') & \text{Lemma~\ref{rookchange}}\\
      &= t^{\alpha}\ (S'''\circ P'') & \text{Proposition~\ref{symmetricgroupcategory}}
  \end{align*}
  In the above equalities, $P''$ is a planar rook diagram, and $S'', S'''$ are permutation diagrams. As in the proof of \cref{rookchange}, $D''':=S'''\circ P''$ is a rook diagram. Thus we have shown that the diagram $D''$ is a power of $t$ times a rook diagram.
\end{proof}

\begin{prop}\label{rookcomb}
  Every morphism in $\Hom_{\cR(t)}(k,l)$ is a $\kk$-linear combination of rook diagrams of type $\binom{l}{k}$.
\end{prop}
\begin{proof}
  The proof, which uses \cref{rookcompose}, is analogous to the proof of \cref{prcomb}.
\end{proof}

\begin{prop}\label{rookindependent}
  All the rook diagrams of type $\binom{l}{k}$ are linearly independent in the morphism space $\Hom_{\cR(t)}(k,l)$.
\end{prop}
\begin{proof}
  We can define a strict $\kk$-linear monoidal functor $G$ from the rook category $\cR(t)$ to the partition category $\Par(t)$ by 
  \[
  1\mapsto 1,\quad s\mapsto s,\quad \eta\mapsto\eta,\quad \varepsilon\mapsto\varepsilon.
  \]
  The functor $G$ is well-defined because relations \cref{R1} to \cref{R3} still hold in $\Par(t)$ (relations \cref{P2} to \cref{P4}). The proof is then analogous to the proof of \cref{prindependent}.
\end{proof}

\begin{theo}\label{rookbasistheorem}
  The morphism space $\Hom_{\cR(t)}(k,l)$ has a linear basis given by all rook diagrams of type $\binom{l}{k}$.
\end{theo}
\begin{proof}
  The proof is a combination of \cref{rookin}, \cref{rookcomb}, and \cref{rookindependent}.
\end{proof}

\begin{cor}
  Fix $k\in\N$. The endomorphism algebra $\End_{\cR(t)}(k)$ is the \emph{rook algebra} $R_k(t)$ given in \cref{introduction}.
\end{cor}
\begin{proof} This is a corollary of \cref{rookbasistheorem}.
\end{proof}
\section{The rook-Brauer category $\RB(t)$ \label{rbcategorysection}}

We first discuss the Brauer category $\cB(t)$ following \cite[\S2]{Leh15}, then we define the rook-Brauer category $\RB(t)$.
\subsection{The Brauer category $\cB(t)$}
\begin{defin}\label{brauercategory}
  Fix $t\in\kk$. The Brauer category $\cB(t)$ is the subcategory of the partition category $\Par(t)$ that has the same objects as $\Par(t)$ and morphism space $\Hom_{\cB(t)}(k,l)$ consisting all formal $\kk$-linear combinations of Brauer diagrams of type $\binom{l}{k}$.
\end{defin}

Fix $k\in\N$. The endomorphism algebra $\End_{\cB(t)}(k)$ is the \emph{Brauer algebra} $B_k(t)$ given in \cref{introduction}. Moreover, the Brauer category $\cB(t)$ has a presentation given in the following theorem.

\begin{theo}[{\cite[Th.~2.6]{Leh15}}]\label{brauercategorypresentation}
  As a strict $\kk$-linear monoidal category, the Brauer category $\cB(t)$ is generated by the object $1$ and the morphisms
  \[
    s =
   \begin{tikzpicture}[anchorbase]
     \pd{0,0};
     \pd{0.7,0};
     \pd{0,0.7};
     \pd{0.7,0.7};
     \draw (0,0) \braidto (0.7,0.7);
     \draw (0.7,0) \braidto (0,0.7);
   \end{tikzpicture}
   \colon 2 \to 2,\quad
    d =
    \begin{tikzpicture}[anchorbase]
      \pd{0,0};
      \pd{0.7,0};
      \draw (0,0) to[out=up,in=up,looseness=2] (0.7,0);
    \end{tikzpicture}
   \colon 2 \to 0,\quad
    c =
    \begin{tikzpicture}[anchorbase]
      \pd{0,0};
      \pd{0.7,0};
      \draw (0,0) to[out=down,in=down,looseness=2] (0.7,0);
    \end{tikzpicture},
  \]
  subject to the following relations and their transforms under $^*$ and $^{\sharp}$:
  \begin{equation}\label{B1}
    \begin{tikzpicture}[anchorbase]
      \pd{0,0};\pd{0.7,0};
      \pd{0,0.7};\pd{0.7,0.7};
      \pd{0,1.4};\pd{0.7,1.4};
      \draw (0,0) \braidto (0.7,0.7) \braidto (0,1.4);
      \draw (0.7,0) \braidto (0,0.7) \braidto (0.7,1.4);
    \end{tikzpicture}
    \ =\
    \begin{tikzpicture}[anchorbase]
      \pd{0,0};\pd{0.7,0};
      \pd{0,0.7};\pd{0.7,0.7};
      \draw (0,0) to (0,0.7);
      \draw (0.7,0) to (0.7,0.7);
    \end{tikzpicture}
    \ ,\qquad
    \begin{tikzpicture}[anchorbase]
      \pd{0,0};\pd{0.7,0};\pd{1.4,0};
      \pd{0,0.7};\pd{0.7,0.7};\pd{1.4,0.7};
      \pd{0,1.4};\pd{0.7,1.4};\pd{1.4,1.4};
      \pd{0,2.1};\pd{0.7,2.1};\pd{1.4,2.1};
      \draw (0,0) to (0,0.7) \braidto (0.7,1.4) \braidto (1.4,2.1);
      \draw (0.7,0) \braidto (1.4,0.7) to (1.4,1.4) \braidto (0.7,2.1);
      \draw (1.4,0) \braidto (0.7,0.7) \braidto (0,1.4) to (0,2.1);
    \end{tikzpicture}
    \ =\
    \begin{tikzpicture}[anchorbase]
      \pd{0,0};\pd{0.7,0};\pd{1.4,0};
      \pd{0,0.7};\pd{0.7,0.7};\pd{1.4,0.7};
      \pd{0,1.4};\pd{0.7,1.4};\pd{1.4,1.4};
      \pd{0,2.1};\pd{0.7,2.1};\pd{1.4,2.1};
      \draw (0,0) \braidto (0.7,0.7) \braidto (1.4,1.4) to (1.4,2.1);
      \draw (0.7,0) \braidto (0,0.7) to (0,1.4) \braidto (0.7,2.1);
      \draw (1.4,0) to (1.4,0.7) \braidto (0.7,1.4) \braidto (0,2.1);
    \end{tikzpicture}
    \ ,
  \end{equation}
  \begin{equation}\label{B2}
    \begin{tikzpicture}[anchorbase]
      \pd{1.4,1.4};
      \pd{0,0.7};\pd{0.7,0.7};\pd{1.4,0.7};
      \pd{0,0};
      \draw (0,0.7) to[out=up,in=up,looseness=2] (0.7,0.7);
      \draw (0.7,0.7) to[out=down,in=down,looseness=2] (1.4,0.7);
      \draw (0,0) to (0,0.7);
      \draw (1.4,0.7) to (1.4,1.4);
    \end{tikzpicture}
    \ =\ 
    \begin{tikzpicture}[anchorbase]
     \pd{0,0.7};
     \pd{0,0};
     \draw (0,0) to (0,0.7);
    \end{tikzpicture}
    \ ,\qquad
    \begin{tikzpicture}[anchorbase]
      \pd{0,1.4};
      \pd{0,0.7};\pd{0.7,0.7};\pd{1.4,0.7};
      \pd{0,0};\pd{0.7,0};\pd{1.4,0};
      \draw (0,0) \braidto (0.7,0.7);
      \draw (0.7,0) \braidto (0,0.7) to (0,1.4);
      \draw (1.4,0) to (1.4,0.7);
      \draw (0.7,0.7) to[out=up,in=up,looseness=2] (1.4,0.7);
    \end{tikzpicture}
    \ =\ 
    \begin{tikzpicture}[anchorbase]
      \pd{1.4,1.4};
      \pd{0,0.7};\pd{0.7,0.7};\pd{1.4,0.7};
      \pd{0,0};\pd{0.7,0};\pd{1.4,0};
      \draw (0,0) to (0,0.7);
      \draw (0.7,0) \braidto (1.4,0.7) to (1.4,1.4);
      \draw (1.4,0) \braidto (0.7,0.7);
      \draw (0,0.7) to[out=up,in=up,looseness=2] (0.7,0.7);
    \end{tikzpicture}
    \ ,
  \end{equation}
  \begin{equation}\label{B3}
    \begin{tikzpicture}[anchorbase]
      \pd{0,0.7};\pd{0.7,0.7};
      \pd{0,0};\pd{0.7,0};
      \draw (0,0) \braidto (0.7,0.7);
      \draw (0.7,0) \braidto (0,0.7);
      \draw (0,0.7) to[out=up,in=up,looseness=2] (0.7,0.7);
    \end{tikzpicture}
    \ =\ 
    \begin{tikzpicture}[anchorbase]
      \pd{0,0};
      \pd{0.7,0};
      \draw (0,0) to[out=up,in=up,looseness=2] (0.7,0);
    \end{tikzpicture}
    \ ,\qquad
    \begin{tikzpicture}[anchorbase]
      \pd{0,0};\pd{0.7,0};
      \draw (0,0) to[out=up,in=up,looseness=2] (0.7,0);
      \draw (0,0) to[out=down,in=down,looseness=2] (0.7,0);
    \end{tikzpicture}
    \ =\ 
    t 1_0
    \ .
  \end{equation}
\end{theo}
Here we use $d$ to denote the cap map, which means ``destroy'', and $c$ to denote the cup map, which means ``create''.

\subsection{The rook-Brauer category $\RB(t)$}
We have defined rook-Brauer diagrams in \cref{introduction}. Now, let us define the rook-Brauer category $\RB(t)$.
\begin{defin}\label{rbcategory}
  Fix $t\in\kk$. We define the rook-Brauer category $\RB(t)$ to be the strict $\kk$-linear monoidal category generated by the object $1$ and the morphisms
  \begin{equation}\label{rbgenerator}
    \begin{split}
    s =
   \begin{tikzpicture}[anchorbase]
     \pd{0,0};
     \pd{0.7,0};
     \pd{0,0.7};
     \pd{0.7,0.7};
     \draw (0,0) \braidto (0.7,0.7);
     \draw (0.7,0) \braidto (0,0.7);
   \end{tikzpicture}
   \colon 2 \to 2,\quad
   & d =
    \begin{tikzpicture}[anchorbase]
      \pd{0,0};
      \pd{0.7,0};
      \draw (0,0) to[out=up,in=up,looseness=2] (0.7,0);
    \end{tikzpicture}
   \colon 2 \to 0,\quad
    c =
    \begin{tikzpicture}[anchorbase]
      \pd{0,0};
      \pd{0.7,0};
      \draw (0,0) to[out=down,in=down,looseness=2] (0.7,0);
    \end{tikzpicture}
   \colon 0 \to 2,\quad\\
   \eta =
   \begin{tikzpicture}[anchorbase]
     \pd{0,0.5};
     \draw(0,0.25) to (0,0.5);
   \end{tikzpicture}
   &\ \colon 0 \to 1,\quad
   \varepsilon =
   \begin{tikzpicture}[anchorbase]
     \pd{0,0};
     \draw (0,0.25) to (0,0);
   \end{tikzpicture}
   \ \colon 1 \to 0,
  \end{split}
  \end{equation}
  subject to the following relations and their transforms under $^*$ and $^{\sharp}$:
  \begin{equation}\label{RB1}
    \begin{tikzpicture}[anchorbase]
      \pd{0,0};\pd{0.7,0};
      \pd{0,0.7};\pd{0.7,0.7};
      \pd{0,1.4};\pd{0.7,1.4};
      \draw (0,0) \braidto (0.7,0.7) \braidto (0,1.4);
      \draw (0.7,0) \braidto (0,0.7) \braidto (0.7,1.4);
    \end{tikzpicture}
    \ =\
    \begin{tikzpicture}[anchorbase]
      \pd{0,0};\pd{0.7,0};
      \pd{0,0.7};\pd{0.7,0.7};
      \draw (0,0) to (0,0.7);
      \draw (0.7,0) to (0.7,0.7);
    \end{tikzpicture}
    \ ,\qquad
    \begin{tikzpicture}[anchorbase]
      \pd{0,0};\pd{0.7,0};\pd{1.4,0};
      \pd{0,0.7};\pd{0.7,0.7};\pd{1.4,0.7};
      \pd{0,1.4};\pd{0.7,1.4};\pd{1.4,1.4};
      \pd{0,2.1};\pd{0.7,2.1};\pd{1.4,2.1};
      \draw (0,0) to (0,0.7) \braidto (0.7,1.4) \braidto (1.4,2.1);
      \draw (0.7,0) \braidto (1.4,0.7) to (1.4,1.4) \braidto (0.7,2.1);
      \draw (1.4,0) \braidto (0.7,0.7) \braidto (0,1.4) to (0,2.1);
    \end{tikzpicture}
    \ =\
    \begin{tikzpicture}[anchorbase]
      \pd{0,0};\pd{0.7,0};\pd{1.4,0};
      \pd{0,0.7};\pd{0.7,0.7};\pd{1.4,0.7};
      \pd{0,1.4};\pd{0.7,1.4};\pd{1.4,1.4};
      \pd{0,2.1};\pd{0.7,2.1};\pd{1.4,2.1};
      \draw (0,0) \braidto (0.7,0.7) \braidto (1.4,1.4) to (1.4,2.1);
      \draw (0.7,0) \braidto (0,0.7) to (0,1.4) \braidto (0.7,2.1);
      \draw (1.4,0) to (1.4,0.7) \braidto (0.7,1.4) \braidto (0,2.1);
    \end{tikzpicture}
    \ ,
  \end{equation}
  \begin{equation}\label{RB2}
    \begin{tikzpicture}[anchorbase]
      \pd{1.4,1.4};
      \pd{0,0.7};\pd{0.7,0.7};\pd{1.4,0.7};
      \pd{0,0};
      \draw (0,0.7) to[out=up,in=up,looseness=2] (0.7,0.7);
      \draw (0.7,0.7) to[out=down,in=down,looseness=2] (1.4,0.7);
      \draw (0,0) to (0,0.7);
      \draw (1.4,0.7) to (1.4,1.4);
    \end{tikzpicture}
    \ =\ 
    \begin{tikzpicture}[anchorbase]
     \pd{0,0.7};
     \pd{0,0};
     \draw (0,0) to (0,0.7);
    \end{tikzpicture}
    \ ,\qquad
    \begin{tikzpicture}[anchorbase]
      \pd{0,1.4};
      \pd{0,0.7};\pd{0.7,0.7};\pd{1.4,0.7};
      \pd{0,0};\pd{0.7,0};\pd{1.4,0};
      \draw (0,0) \braidto (0.7,0.7);
      \draw (0.7,0) \braidto (0,0.7) to (0,1.4);
      \draw (1.4,0) to (1.4,0.7);
      \draw (0.7,0.7) to[out=up,in=up,looseness=2] (1.4,0.7);
    \end{tikzpicture}
    \ =\ 
    \begin{tikzpicture}[anchorbase]
      \pd{1.4,1.4};
      \pd{0,0.7};\pd{0.7,0.7};\pd{1.4,0.7};
      \pd{0,0};\pd{0.7,0};\pd{1.4,0};
      \draw (0,0) to (0,0.7);
      \draw (0.7,0) \braidto (1.4,0.7) to (1.4,1.4);
      \draw (1.4,0) \braidto (0.7,0.7);
      \draw (0,0.7) to[out=up,in=up,looseness=2] (0.7,0.7);
    \end{tikzpicture}
    \ ,
  \end{equation}
  \begin{equation}\label{RB3}
    \begin{tikzpicture}[anchorbase]
      \pd{0,0.7};\pd{0.7,0.7};
      \pd{0,0};\pd{0.7,0};
      \draw (0,0) \braidto (0.7,0.7);
      \draw (0.7,0) \braidto (0,0.7);
      \draw (0,0.7) to[out=up,in=up,looseness=2] (0.7,0.7);
    \end{tikzpicture}
    \ =\ 
    \begin{tikzpicture}[anchorbase]
      \pd{0,0};
      \pd{0.7,0};
      \draw (0,0) to[out=up,in=up,looseness=2] (0.7,0);
    \end{tikzpicture}
    \ ,\qquad
   \begin{tikzpicture}[anchorbase]
     \pd{0.7,0};
     \pd{0,0.7};
     \pd{0.7,0.7};
     \pd{0,1.4};
     \pd{0.7,1.4};
     \draw (0,0.45) to (0,0.7) \braidto (0.7,1.4);
     \draw (0.7,0) to (0.7,0.7) \braidto (0,1.4);
   \end{tikzpicture}
   \ =\
   \begin{tikzpicture}[anchorbase]
     \pd{0,0};
     \pd{0,0.7};
     \pd{0.7,0.7};
     \draw (0,0) to (0,0.7);
     \draw (0.7,0.45) to (0.7,0.7);
   \end{tikzpicture}
   \ ,
 \end{equation}
 \begin{equation}\label{RB4}
   \begin{tikzpicture}[anchorbase]
     \pd{0,0};\pd{0.7,0};
     \draw (0,0) to[out=up,in=up,looseness=2] (0.7,0);
     \draw (0,0) to[out=down,in=down,looseness=2] (0.7,0);
   \end{tikzpicture}
   \ =\ 
   \begin{tikzpicture}[anchorbase]
    \pd{0,0};\pd{0.7,0};
    \draw (0,-0.15) to (0,0);
    \draw (0.7,-0.15) to (0.7,0);
    \draw (0,0) to[out=up,in=up,looseness=2] (0.7,0);
   \end{tikzpicture}
   \ =\ 
   \begin{tikzpicture}[anchorbase]
     \pd{0,0};
     \draw (0,-0.2) to (0,0.2);
   \end{tikzpicture}
   \ =\ 
   t 1_0\ .
 \end{equation}
\end{defin}

\begin{prop} The cap map $d$ can pass through the braiding $s$ in the rook-Brauer category $\RB(t)$.
  \begin{equation}\label{RB5}
    \begin{tikzpicture}[anchorbase]
      \pd{0,2.1};
      \pd{0,1.4};\pd{0.7,1.4};\pd{1.4,1.4};
      \pd{0,0.7};\pd{0.7,0.7};\pd{1.4,0.7};
      \pd{0,0};\pd{0.7,0};\pd{1.4,0};
      \draw (0,0) to (0,0.7) \braidto (0.7,1.4);
      \draw (0.7,0) \braidto (1.4,0.7) to (1.4,1.4);
      \draw (1.4,0) \braidto (0.7,0.7) \braidto (0,1.4) to (0,2.1);
      \draw (0.7,1.4) to[out=up,in=up,looseness=2] (1.4,1.4);
    \end{tikzpicture}
    \ =\ 
    \begin{tikzpicture}[anchorbase]
      \pd{1.4,0.7};
      \pd{0,0};\pd{0.7,0};\pd{1.4,0};
      \draw (0,0) to[out=up,in=up,looseness=2] (0.7,0);
      \draw (1.4,0) to (1.4,0.7);
    \end{tikzpicture}
    \ ,
  \end{equation}
\end{prop}

\begin{proof}
  \[
    \begin{tikzpicture}[anchorbase]
      \pd{0,2.1};
      \pd{0,1.4};\pd{0.7,1.4};\pd{1.4,1.4};
      \pd{0,0.7};\pd{0.7,0.7};\pd{1.4,0.7};
      \pd{0,0};\pd{0.7,0};\pd{1.4,0};
      \draw (0,0) to (0,0.7) \braidto (0.7,1.4);
      \draw (0.7,0) \braidto (1.4,0.7) to (1.4,1.4);
      \draw (1.4,0) \braidto (0.7,0.7) \braidto (0,1.4) to (0,2.1);
      \draw (0.7,1.4) to[out=up,in=up,looseness=2] (1.4,1.4);
    \end{tikzpicture}
    \ \overset{\cref{RB2}}{=\joinrel=}\
    \begin{tikzpicture}[anchorbase]
      \pd{1.4,2.1};
      \pd{0,1.4};\pd{0.7,1.4};\pd{1.4,1.4};
      \pd{0,0.7};\pd{0.7,0.7};\pd{1.4,0.7};
      \pd{0,0};\pd{0.7,0};\pd{1.4,0};
      \draw (0,0) to (0,0.7) to (0,1.4);
      \draw (0.7,0) \braidto (1.4,0.7) \braidto (0.7,1.4);
      \draw (1.4,0) \braidto (0.7,0.7) \braidto (1.4,1.4) to (1.4,2.1);
      \draw (0,1.4) to[out=up,in=up,looseness=2] (0.7,1.4);
    \end{tikzpicture}
    \ \overset{\cref{RB1}}{=\joinrel=}\
    \begin{tikzpicture}[anchorbase]
      \pd{1.4,1.4};
      \pd{0,0.7};\pd{0.7,0.7};\pd{1.4,0.7};
      \pd{0,0};\pd{0.7,0};\pd{1.4,0};
      \draw (0,0) to (0,0.7);
      \draw (0.7,0) to (0.7,0.7);
      \draw (1.4,0) to (1.4,0.7) to (1.4,1.4);
      \draw (0,0.7) to[out=up,in=up,looseness=2] (0.7,0.7);
    \end{tikzpicture}
    \ \overset{\cref{ax4}}{=\joinrel=}\
    \begin{tikzpicture}[anchorbase]
      \pd{1.4,0.7};
      \pd{0,0};\pd{0.7,0};\pd{1.4,0};
      \draw (0,0) to[out=up,in=up,looseness=2] (0.7,0);
      \draw (1.4,0) to (1.4,0.7);
    \end{tikzpicture}
    \ ,
  \]
\end{proof}

We will show that morphism spaces of the rook-Brauer category $\RB(t)$ have linear bases given by all rook-Brauer diagrams.

\begin{prop}\label{rbin}
  Every rook-Brauer diagram can be built by compositions and tensor products of the generators $\{s,d,c,\eta,\varepsilon\}$ given in \cref{rbgenerator}.
\end{prop}
\begin{proof}
  The proof, which uses \cref{decompose2}\cref{decompose2b}, is analogous to the proof of \cref{rookin}.
\end{proof}

\begin{lem}\label{rbchange}
  Consider a diagram $B\circ P$, where $B$ is a Brauer diagram, and $P$ is a planar rook diagram. Then there exists a Brauer diagram $B'$, a planar rook diagram $P'$, and a nonnegative integer $\alpha$ such that $B\circ P=t^{\alpha}P'\circ B'$.
\end{lem}
\begin{proof}
  This proof is similar to the proof of \cref{rookchange}. By relations \cref{RB3} and \cref{RB4}, there exists a nonnegative integer $\alpha$ and a rook-Brauer diagram $D$ such that $B\circ P=t^{\alpha}D$. Then by \cref{decompose2cor}\cref{decompose2corb}, there exists a Brauer diagram $B'$ and a planar rook diagram $P'$ such that $D=P'\circ B'$. Therefore, we have $B\circ P=t^{\alpha}D=P'\circ B'$, which finishes the proof.
\end{proof}

\begin{lem}\label{rbcompose}
  Let $D\colon l\to m$ and $D'\colon k\to l$ be two rook-Brauer diagrams. Then the diagram $D'':=D\circ D'$ is equal to a power of $t$ times a rook-Brauer diagram.
\end{lem}
\begin{proof}
  The proof, which uses \cref{decompose2}\cref{decompose2b}, \cref{decompose2cor}\cref{decompose2corb}, and \cref{rbchange}, is similar to the proof of \cref{rookcompose}.
\end{proof}

\begin{prop}\label{rbcomb}
  Every morphism in $\Hom_{\RB(t)}(k,l)$ is a $\kk$-linear combination of rook-Brauer diagrams of type $\binom{l}{k}$.
\end{prop}
\begin{proof}
  The proof, which uses \cref{rbcompose}, is analogous to the proof of \cref{prcompose}.
\end{proof}

\begin{prop}\label{rbindependent}
  All the rook-Brauer diagrams of type $\binom{l}{k}$ are linearly independent in the morphism space $\Hom_{\RB(t)}(k,l)$.
\end{prop}
\begin{proof}
  We can define a strict $\kk$-linear monoidal functor $H$ from the rook-Brauer category $\RB(t)$ to the partition category $\Par(t)$ by
  \[
  1\mapsto 1,\quad s\mapsto s,\quad n\mapsto n,\quad u\mapsto u,\quad\eta\mapsto\eta,\quad \varepsilon\mapsto\varepsilon.
  \]
  The functor $H$ is well-defined because relations \cref{RB1} to \cref{RB4} hold in $\Par(t)$ (relations \cref{P1} to \cref{P4}). The proof is then analogous to the proof of \cref{prindependent}.
\end{proof}

\begin{theo}\label{rbbasistheorem}
  The morphism space $\Hom_{\RB(t)}(k,l)$ has a linear basis given by all rook-Brauer diagrams of type $\binom{l}{k}$.
\end{theo}
\begin{proof}The proof is a combination of \cref{rbin}, \cref{rbcomb}, and \cref{rbindependent}.
\end{proof}

\begin{cor} Fix $k\in\N$. The endomorphism algebra $\End_{\RB(t)}(k)$ is the \emph{rook-Brauer algebra} given in \cref{introduction}.
\end{cor}
\begin{proof}
  This is a corollary of \cref{rbbasistheorem}.
\end{proof}

\section{The Motzkin category $\cM(t)$\label{motzkincategorysection}}

We first discuss the Temperley-Lieb category $\TL(t)$ following \cite{Che14}, then we define the Motzkin category $\cM(t)$.

\subsection{The Temperley-Lieb category $\TL(t)$}
\begin{defin}\label{tlcategory}
  Fix $t\in\kk$, the Temperley-Lieb category $\TL(t)$ is the subcategory of the partition category $\Par(t)$ that has the same objects as $\Par(t)$ and morphism space $\Hom_{\TL(t)}(k,l)$ consisting all formal $\kk$-linear combinations of Temperley-Lieb diagrams of type $\binom{l}{k}$.
\end{defin}

Fix $k\in\N$. The endomorphism algebra $\End_{\TL(t)}(k)$ is the \emph{Temperley-Lieb algebra} $TL_k(t)$ given in \cref{introduction}.

\begin{theo}\label{tlbasistheorem}
  As a strict $\kk$-linear monoidal category, the Temperley-Lieb category $\TL(t)$ is generated by the object $1$ and the morphisms
  \[
    d\ =\ 
    \begin{tikzpicture}[anchorbase]
      \pd{0,0};
      \pd{0.7,0};
      \draw (0,0) to[out=up,in=up,looseness=2] (0.7,0);
    \end{tikzpicture}
    \colon 2 \to 0, \quad 
    c\ =\ 
    \begin{tikzpicture}[anchorbase]
      \pd{0,0};
      \pd{0.7,0};
      \draw (0,0) to[out=down,in=down,looseness=2] (0.7,0);
    \end{tikzpicture}
   \colon 0 \to 2,\quad
  \]
  subject to the following relations:
  \[
    \begin{tikzpicture}[anchorbase]
      \draw (0,0.7) to (0,0) to[out=down,in=down,looseness=1.6] (0.7,0) to[out=up,in=up,looseness=1.8] (1.4,0) to (1.4,-0.7);
    \end{tikzpicture}
    \ =\ 
    \begin{tikzpicture}[anchorbase]
      \draw (0,0) to (0,0.7);
    \end{tikzpicture}
    \ =\ 
    \begin{tikzpicture}[anchorbase]
      \draw (0,-0.7) to (0,0) to[out=up,in=up,looseness=1.6] (0.7,0) to[out=down,in=down,looseness=1.8] (1.4,0) to (1.4,0.7);
    \end{tikzpicture}
    \ ,\qquad
    \begin{tikzpicture}[anchorbase]
      \draw (0,0) to[out=up,in=up,looseness=1.8] (0.7,0) to[out=down,in=down,looseness=1.8] (0,0);
    \end{tikzpicture}
    \ =\ 
    t1_{\one}
  \]
\end{theo}
\begin{proof}
  The proof is similar to the proof of \cite[Th.~4.3]{Kau90}, which gives a presentation of the Temperley-Lieb algebras. Alternatively, we can prove this theorem using a method similar to the one used in the proof of \cref{prbasistheorem}.
\end{proof}

\subsection{The Motzkin category $\cM(t)$}
We have defined Motzkin diagrams in \cref{introduction}. Now, let us define the Motzkin category $\cM(t)$.

\begin{defin}\label{motzkincategory}
  Fix $t\in\kk$. We define the Motzkin category $\cM(t)$ to be the strict $\kk$-linear monoidal category generated by the object $1$ and the morphisms
  \begin{equation}\label{motzkingenerator}
    d\ =\ 
    \begin{tikzpicture}[anchorbase]
      \pd{0,0};
      \pd{0.7,0};
      \draw (0,0) to[out=up,in=up,looseness=2] (0.7,0);
    \end{tikzpicture}
    \colon 2 \to 0, \quad 
    c\ =\ 
    \begin{tikzpicture}[anchorbase]
      \pd{0,0};
      \pd{0.7,0};
      \draw (0,0) to[out=down,in=down,looseness=2] (0.7,0);
    \end{tikzpicture}
   \colon 0 \to 2,\quad
   \eta =
   \begin{tikzpicture}[anchorbase]
     \pd{0,0.5};
     \draw(0,0.25) to (0,0.5);
   \end{tikzpicture}
   \ \colon 0 \to 1,\quad
   \varepsilon =
   \begin{tikzpicture}[anchorbase]
     \pd{0,0};
     \draw (0,0.25) to (0,0);
   \end{tikzpicture}
   \ \colon 1 \to 0,
  \end{equation}
  subject to the following relations and their transforms under $^*$ and $^{\sharp}$:
  \begin{equation}\label{M1}
    \begin{tikzpicture}[anchorbase]
      \pd{0,0};\pd{0.7,0};
      \pd{0,-0.7};
      \draw (0,-0.7) to (0,0) to[out=up,in=up,looseness=2] (0.7,0) to (0.7,-0.15);
    \end{tikzpicture}
    \ =\ 
    \begin{tikzpicture}[anchorbase]
      \pd{0,0};
      \draw (0,0) to (0,-0.15);
    \end{tikzpicture}
    \ ,\qquad
    \begin{tikzpicture}[anchorbase]
      \pd{1.4,1.4};
      \pd{0,0.7};\pd{0.7,0.7};\pd{1.4,0.7};
      \pd{0,0};
      \draw (0,0.7) to[out=up,in=up,looseness=2] (0.7,0.7);
      \draw (0.7,0.7) to[out=down,in=down,looseness=2] (1.4,0.7);
      \draw (0,0) to (0,0.7);
      \draw (1.4,0.7) to (1.4,1.4);
    \end{tikzpicture}
    \ =\ 
    \begin{tikzpicture}[anchorbase]
     \pd{0,0.7};
     \pd{0,0};
     \draw (0,0) to (0,0.7);
    \end{tikzpicture}
    \ ,
  \end{equation}
  \begin{equation}\label{M2}
  \begin{tikzpicture}[anchorbase]
    \pd{0,0};\pd{0.7,0};
    \draw (0,0) to[out=up,in=up,looseness=1.8] (0.7,0) to[out=down,in=down,looseness=1.8] (0,0);
  \end{tikzpicture}
  \ =\ 
  \begin{tikzpicture}[anchorbase]
    \pd{0,0};
    \draw (0,-0.15) to (0,0) to (0,0.15);
  \end{tikzpicture}
  \ =\ 
  t1_0
  \ .
\end{equation}
\end{defin}

\begin{prop}\label{motzkinin}
  Every Motzkin diagram can be built by compositions and tensor products of the generators $\{d,c,\eta,\varepsilon\}$ given in \cref{motzkingenerator}. 
\end{prop}
\begin{proof}
  The proof, which uses \cref{partitiondecomposecor}\cref{decompose1c}, is analogous to the proof of \cref{prin}.
\end{proof}

\begin{prop}\label{motzkincompose}
  Let $D\colon l\to m$ and $D'\colon k\to l$ be two Motzkin diagrams. Then the diagram $D'':=D\circ D'$ is equal to a power of $t$ times a Motzkin diagram.
\end{prop}
\begin{proof}
  Apart from the four steps \cref{prreplacea} to \cref{prreplaced} given in the proof of \cref{prcompose}, we add five more local substitutions and their transforms under $^*$ and $^{\sharp}$.
\begin{enumerate}
  \renewcommand{\theenumi}{(\alph{enumi})}
  \setcounter{enumi}{4}
  \item\label{replacee} replace every
  \[
    \ \ \begin{tikzpicture}[anchorbase]
      \pd{0,1};\pd{0.7,1};
      \pd{0,0};\pd{0.7,0};
      \draw (0,1) to (0,0) to[out=down,in=down,looseness=1.8] (0.7,0) to (0.7,1);
    \end{tikzpicture}\ \ 
    \text{with}
    \ \ \begin{tikzpicture}[anchorbase]
      \pd{0,0};\pd{0.7,0};
      \draw (0,0) to[out=down,in=down,looseness=1.8] (0.7,0);
    \end{tikzpicture}\ ,
  \]
\item\label{replacef} replace every
\[
  \ \ \begin{tikzpicture}[anchorbase]
    \pd{0,1};
    \pd{0,0};\pd{0.7,0};
    \draw (0,1) to (0,0) to[out=down,in=down,looseness=1.8] (0.7,0) to (0.7,0.15);
  \end{tikzpicture}\ \ 
  \text{with}
  \ \ \begin{tikzpicture}[anchorbase]
    \pd{0,0};
    \draw (0,0) to (0,-0.15);
  \end{tikzpicture}\ ,
\]
\item\label{replaceg} replace every
\[
  \ \ \begin{tikzpicture}[anchorbase]
    \pd{0,0.7};
    \pd{0,0};\pd{0.7,0};\pd{1.4,0};
    \pd{1.4,-0.7};
    \draw (0,0.7) to (0,0) to[out=down,in=down,looseness=1.8] (0.7,0) to[out=up,in=up,looseness=1.8] (1.4,0) to (1.4,-0.7);
  \end{tikzpicture}\ \ 
  \text{with}
  \ \ \begin{tikzpicture}[anchorbase]
    \pd{0,0.7};
    \pd{0,0};
    \draw (0,0) to (0,0.7);
  \end{tikzpicture}\ ,
\]
\item\label{replaceh} replace every
\[
  \ \ \begin{tikzpicture}[anchorbase]
    \pd{0,0.7};
    \pd{0,0};\pd{0.7,0};\pd{1.4,0};
    \draw (0,0.7) to (0,0) to[out=down,in=down,looseness=1.8] (0.7,0) to[out=up,in=up,looseness=1.8] (1.4,0) to (1.4,-0.15);
  \end{tikzpicture}\ \ 
  \text{with}
  \ \ \begin{tikzpicture}[anchorbase]
    \pd{0,0};
    \draw (0,0) to (0,-0.15);
  \end{tikzpicture},\ 
\]
\item replace every bubble with a scalar multiple $t$.
\end{enumerate}
After steps \cref{prreplacea} to \cref{replaceh}, what we obtain is a power of $t$ times a Motzkin diagram.
\end{proof}

\begin{prop}\label{motzkincomb}
  Every morphism in $\Hom_{\cM(t)}(k,l)$ is a $\kk$-linear combination of Motzkin diagrams of type $\binom{l}{k}$.
\end{prop}
\begin{proof}
  The proof, which uses \cref{motzkincompose}, is analogous to the proof of \cref{prcomb}.
\end{proof}

\begin{prop}\label{motzkinindependent}
  All the Motzkin diagrams of type $\binom{l}{k}$ are linearly independent in $\Hom_{\cM(t)}(k,l)$.
\end{prop}
\begin{proof}
  We can define a strict $\kk$-linear monoidal functor $J\colon\cM(t)\to\Par(t)$ by
  \[
    1\mapsto 1,\quad d\mapsto d,\quad c\mapsto c,\quad \eta\mapsto\eta,\quad \varepsilon\mapsto\varepsilon.
  \]
  The functor $J$ is well-defined because relations \cref{M1} to \cref{M2} hold in $\Par(t)$ (relations \cref{P1,P3,P4}). The proof is then analogous to the proof of \cref{prindependent}.
\end{proof}

\begin{theo}\label{motzkinbasistheorem}
  The morphism space $\Hom_{\cM(t)}(k,l)$ has a linear basis given by all Motzkin diagrams of type $\binom{l}{k}$.
\end{theo}
\begin{proof}
  The proof is a combination of \cref{motzkinin}, \cref{motzkincomb}, and \cref{motzkinindependent}.
\end{proof}
\begin{cor} Fix $k\in\N$. The endomorphism algebra $\End_{\cM(t)}(k)$ is the \emph{Motzkin algebra} $M_k(t)$ given in \cref{introduction}.
\end{cor}
\begin{proof}
  This is a corollary of \cref{motzkinbasistheorem}.
\end{proof}

\begin{rem}
  \begin{enumerate}
    \item Since the Motzkin category $\cM(t)$ can also be viewed as the planar rook-Brauer category, we can organize the last two sections of this paper in another way. Namely, we can exchange \cref{rbcategorysection} and \cref{motzkincategorysection}. In this case, we can use the decomposition given in \cref{decompose2}\cref{decompose2c} instead of \cref{decompose2b} to prove \cref{rbin}.
    \item Because morphisms in the categories $\PR(t)$ and $\cM(t)$ are planar partition diagrams, the proofs of \cref{prcompose} and \cref{motzkincompose} are very straightforward, using only \cref{prcategory} and \cref{motzkincategory}. However, for the categories $\cR(t)$ and $\RB(t)$, whose morphisms are not necessarily planar, the proofs of \cref{rookcompose} and \cref{rbcompose} are more complicated, involving the decompositions given in \cref{decompose2} and \cref{decompose2cor}.
  \end{enumerate}
\end{rem}

\bibliographystyle{alphaurl}
\bibliography{ref}

\begin{thebibliography}{{Che}14}

\bibitem[BH14]{Ben14}
Georgia Benkart and Tom Halverson.
\newblock Motzkin algebras.
\newblock {\em European J. Combin.}, 36:473--502, 2014.
\newblock \href {http://dx.doi.org/10.1016/j.ejc.2013.09.010}
  {\path{doi:10.1016/j.ejc.2013.09.010}}.

\bibitem[{Che}14]{Che14}
Joshua {Chen}.
\newblock {The Temperley-Lieb categories and skein modules}.
\newblock 2014.
\newblock \href {http://arxiv.org/abs/1502.06845} {\path{arXiv:1502.06845}}.

\bibitem[Com16]{Com16}
Jonathan Comes.
\newblock Jellyfish partition categories.
\newblock {\em Algebr. Represent. Theory}, pages 1--21, 2016.
\newblock \href {http://dx.doi.org/10.1007/s10468-018-09851-7}
  {\path{doi:10.1007/s10468-018-09851-7}}.

\bibitem[Del07]{Del07}
P.~Deligne.
\newblock La cat\'{e}gorie des repr\'{e}sentations du groupe sym\'{e}trique
  {$S_t$}, lorsque {$t$} n'est pas un entier naturel.
\newblock In {\em Algebraic groups and homogeneous spaces}, volume~19 of {\em
  Tata Inst. Fund. Res. Stud. Math.}, pages 209--273. Tata Inst. Fund. Res.,
  Mumbai, 2007.
\newblock URL: \url{http://www.math.ias.edu/files/deligne/Symetrique.pdf}.

\bibitem[Hd14]{Hal14}
Tom Halverson and Elise delMas.
\newblock Representations of the {R}ook-{B}rauer algebra.
\newblock {\em Comm. Algebra}, 42(1):423--443, 2014.
\newblock \href {http://dx.doi.org/10.1080/00927872.2012.716120}
  {\path{doi:10.1080/00927872.2012.716120}}.

\bibitem[Her06]{Her06}
Kathryn~E Herbig.
\newblock The planar rook monoid.
\newblock 2006.
\newblock URL: \url{https://digitalcommons.macalester.edu/mathcs_honors/1/}.

\bibitem[HR05]{Hal05}
Tom Halverson and Arun Ram.
\newblock Partition algebras.
\newblock {\em European J. Combin.}, 26(6):869--921, 2005.
\newblock \href {http://dx.doi.org/10.1016/j.ejc.2004.06.005}
  {\path{doi:10.1016/j.ejc.2004.06.005}}.

\bibitem[Jon94]{Jon94}
V.~F.~R. Jones.
\newblock The {P}otts model and the symmetric group.
\newblock In {\em Subfactors ({K}yuzeso, 1993)}, pages 259--267. World Sci.
  Publ., River Edge, NJ, 1994.

\bibitem[Kas95]{Kas95}
Christian Kassel.
\newblock {\em Quantum groups}, volume 155 of {\em Graduate Texts in
  Mathematics}.
\newblock Springer-Verlag, New York, 1995.
\newblock \href {http://dx.doi.org/10.1007/978-1-4612-0783-2}
  {\path{doi:10.1007/978-1-4612-0783-2}}.

\bibitem[Kau90]{Kau90}
Louis~H. Kauffman.
\newblock An invariant of regular isotopy.
\newblock {\em Trans. Amer. Math. Soc.}, 318(2):417--471, 1990.
\newblock \href {http://dx.doi.org/10.2307/2001315}
  {\path{doi:10.2307/2001315}}.

\bibitem[KM06]{Kud06}
Ganna Kudryavtseva and Volodymyr Mazorchuk.
\newblock On presentations of {B}rauer-type monoids.
\newblock {\em Cent. Eur. J. Math.}, 4(3):413--434, 2006.
\newblock \href {http://dx.doi.org/10.2478/s11533-006-0017-6}
  {\path{doi:10.2478/s11533-006-0017-6}}.

\bibitem[Koc04]{Koc04}
Joachim Kock.
\newblock {\em Frobenius algebras and 2{D} topological quantum field theories},
  volume~59 of {\em London Mathematical Society Student Texts}.
\newblock Cambridge University Press, Cambridge, 2004.
\newblock URL: \url{http://mat.uab.es/~kock/TQFT.html}.

\bibitem[Lip96]{Lip96}
Stephen Lipscomb.
\newblock {\em Symmetric inverse semigroups}, volume~46 of {\em Mathematical
  Surveys and Monographs}.
\newblock American Mathematical Society, Providence, RI, 1996.
\newblock \href {http://dx.doi.org/10.1090/surv/046}
  {\path{doi:10.1090/surv/046}}.

\bibitem[{Liu}18]{Liu18}
Bingyan {Liu}.
\newblock {Presentations of linear monoidal categories and their endomorphism
  algebras}.
\newblock 2018.
\newblock \href {http://arxiv.org/abs/1810.10988} {\path{arXiv:1810.10988}}.

\bibitem[LZ15]{Leh15}
G.~I. Lehrer and R.~B. Zhang.
\newblock The {B}rauer category and invariant theory.
\newblock {\em J. Eur. Math. Soc. (JEMS)}, 17(9):2311--2351, 2015.
\newblock \href {http://dx.doi.org/10.4171/JEMS/558}
  {\path{doi:10.4171/JEMS/558}}.

\bibitem[Mar94]{Mar94}
Paul Martin.
\newblock Temperley-{L}ieb algebras for nonplanar statistical mechanics---the
  partition algebra construction.
\newblock {\em J. Knot Theory Ramifications}, 3(1):51--82, 1994.
\newblock \href {http://dx.doi.org/10.1142/S0218216594000071}
  {\path{doi:10.1142/S0218216594000071}}.

\bibitem[NSR19]{Sav19}
Samuel {Nyobe Likeng}, Alistair {Savage}, and appendix with~Christopher {Ryba}.
\newblock {Embedding Deligne's category $\mathrm{Rep}(S_t)$ in the Heisenberg
  category}.
\newblock 2019.
\newblock \href {http://arxiv.org/abs/1905.05620} {\path{arXiv:1905.05620}}.

\bibitem[{Sav}18]{Sav18}
Alistair {Savage}.
\newblock {String diagrams and categorification}.
\newblock In {\em Interactions of Quantum Affine Algebras with Cluster
  Algebras, Current Algebras and Categorification}, Progress in Mathematics.
  Birkh\"{a}user, 2018.
\newblock To appear.
\newblock \href {http://arxiv.org/abs/1806.06873} {\path{arXiv:1806.06873}}.

\bibitem[Sch01]{Sch01}
Peter Schauenburg.
\newblock Turning monoidal categories into strict ones.
\newblock {\em New York J. Math.}, 7:257--265, 2001.
\newblock URL: \url{http://nyjm.albany.edu/j/2001/7-16.pdf}.

\bibitem[TV17]{Tur17}
Vladimir Turaev and Alexis Virelizier.
\newblock {\em Monoidal categories and topological field theory}, volume 322 of
  {\em Progress in Mathematics}.
\newblock Birkh\"{a}user/Springer, Cham, 2017.
\newblock \href {http://dx.doi.org/10.1007/978-3-319-49834-8}
  {\path{doi:10.1007/978-3-319-49834-8}}.

\end{thebibliography}

\end{document}